\documentclass{article}
\usepackage{amsmath, amsthm, amssymb,bbm,amsfonts,appendix,tikz,graphicx,enumerate}
\def\authorsPS{}
\newcommand\authPS[8]{\ifnum\authPSc<1\def\authPSc{1}\else, \fi{\large\sffamily #1 #2}
\edef\authorsPS{\authorsPS \par\vskip 1mm\noindent #1 #2:\hskip 5mm #3, #4, #5, #6, #7, #8}}

\numberwithin{equation}{section}
\newtheorem{theorem}{Theorem}[section]
\newtheorem{lemma}[theorem]{Lemma}

\theoremstyle{definition}
\newtheorem{definition}[theorem]{Definition}
\theoremstyle{remark}
\newtheorem{remark}[theorem]{Remark}

\def\authPSc{0}
\newcommand{\beq}[1]{
\begin{equation}\label{#1}}
\newcommand{\eeq}{\end{equation}}
\newcommand{\req}[1]{{\rm(\ref{#1})}}
\newcommand{\hten}{{\mathfrak H}}

\newcommand{\Nt}{\lfloor nt \rfloor}
\newcommand{\DXj}{\Delta X_{\frac jn}}

\newcommand{\Indi}[1]{\mathbbm{1}_{#1}}
\newcommand{\floor}[1]{\left\lfloor#1\right\rfloor}
\newcommand{\Ip}[1]{\left\langle #1 \right\rangle}
\newcommand{\Comb}[1]{\left(\begin{array}{c}#1\end{array}\right)}
\newcommand{\Norm}[1]{\left\lVert#1\right\rVert}
\newcommand{\Abs}[1]{\left|#1\right|}
\newcommand{\Cc}{\mathcal{C}}
\newcommand{\Fc}{\mathcal{F}}
\newcommand{\Gc}{\mathcal{G}}
\newcommand{\Hc}{\mathcal{H}}
\newcommand{\Lc}{\mathcal{L}}
\newcommand{\N}{\mathbb{N}}

\newcommand{\E}{\mathbb{E}}
\newcommand{\R}{\mathbb{R}}

\newcommand{\Pb}{\mathbb{P}}
\newcommand{\Hg}{\mathfrak{H}}

\begin{document}

% !!!!!!!!!!!!!!!!!!!!!!!!!!!!!!!!!!!!!!!!!!!!!!!!!!!!!!!!!!!!!!!!!!!!!!!
% !!!                   START HERE                                   !!!
% !!!!!!!!!!!!!!!!!!!!!!!!!!!!!!!!!!!!!!!!!!!!!!!!!!!!!!!!!!!!!!!!!!!!!!!

%% !!  FULL TITLE OF PAPER
% example :
\title{Symmetric stochastic integrals with respect to a class of self-similar Gaussian processes}
\author{Daniel Harnett, Arturo Jaramillo, David Nualart \thanks{D. Nualart was supported by the NSF grant  DMS1512891.} }
\date{\today}
 
 \maketitle

\begin{abstract}
We study the asymptotic behavior of the $\nu$-symmetric Riemman sums for functionals of a self-similar centered Gaussian process $X$ with increment exponent $0<\alpha<1$. We prove that, under mild assumptions on the covariance of $X$, the law of the weak $\nu$-symmetric Riemman sums converge in the Skorohod topology when $\alpha=(2\ell+1)^{-1}$, where $\ell$ denotes the smallest positive integer satisfying 
$\int_{0}^{1}x^{2j}\nu(dx)=(2j+1)^{-1}$ for all $j=0,\dots, \ell-1$. In the case $\alpha>(2\ell+1)^{-1}$, we prove that the convergence holds in probability. 
\end{abstract}


\section{Introduction}\label{Sec:intro}
Consider a centered self-similar Gaussian process $X:=\{X_{t}\}_{t\geq0}$ with self-similarity exponent $\beta \in (0,1)$ defined on a probability space $(\Omega, \mathcal{F}, \mathbb{P})$. That is, $X$ is a centered Gaussian process such that $\{c^{-\beta}X_{ct}\}_{t\geq0}$ has the same law as $X$, for every $c>0$.  We also assume that $X_0=0$. The covariance of $X$ is characterized by the values of the function $\phi:[1,\infty)\rightarrow\R$, defined by 
\begin{align}\label{eq:phicov}
\phi(x)
  &:=\E\left[X_{1}X_{x}\right].
\end{align}
Indeed, for $0< s\leq t$, 
\beq{eq:R} 
R(s,t) := \mathbb{E}\left[ X_s X_t\right] = s^{2\beta}\phi(t/s).
\eeq  
The idea of describing a self-similar Guassian process in terms of the function $\phi$ was first used by Harnett and Nualart in \cite{HN_14}, and the concept was further  developed in \cite{Hermite}. 

The purpose of this paper is to study the  behavior as $n\rightarrow\infty$ of   $\nu$-{\it symmetric Riemann sums} with respect to $X$, defined by 
\beq{nu_integral}
S^\nu_n(g, t) := \sum_{j=0}^{\Nt-1}\int_0^1 g(X_{\frac jn} + y \DXj) \DXj \nu(dy),
\eeq
where $\Delta X_{\frac{j}{n}} = X_{\frac{j+1}{n}} - X_{\frac{j}{n}}$, $g:\mathbb{R}\to\mathbb{R}$ is a sufficiently smooth function and $\nu$ is a  symmetric probability measure on $[0,1]$, meaning that $\nu(A)=\nu(1-A)$ for any Borel set $A\subset[0,1]$. 

 The best known self-similar centered Gaussian process is the fractional Brownian motion (fBm) of Hurst parameter $H\in (0,1)$, whose covariance is given by
\begin{equation} \label{cov}
R(s,t) = \frac{1}{2}\left( t^{2H}+s^{2H}-\Abs{t-s}^{2H}\right).
\end{equation}
The $\nu$-symmetric Riemann sums $S^\nu_n(g, t)$ given in (\ref{nu_integral}) were investigated in the seminal paper by Gradinaru, Nourdin, Russo and Vallois  \cite{GNRV},   when $X$ is a fBm with Hurst parameter $H$. In this case, if $g$ is a function of the form $g=f^{\prime}$ with $f\in \mathcal{C}^{4\ell(\nu)+2}(\R)$ and $\ell=\ell(\nu) \ge 1$ denotes the largest integer such that
\[ 
\int_0^1 \alpha^{2j} \nu(d\alpha) = \frac{1}{2j+1},\ \ \ \ \text{for}\ \ \ \  j = 1, \dots, \ell-1,
\]
then, provided that $H>\frac{1}{4\ell+2}$, there exists a random variable $\int_{0}^{t}g(X_{s})d^{\nu}X_{s}$ such that
\[
S^\nu_n(g, t)\stackrel{\Pb}{\longrightarrow}\int_{0}^{t}g(X_{s})d^{\nu}X_{s}\ \ \ \ \ \ \ \text{as }\ n\rightarrow\infty.
\]
The limit in the right-hand side is called  the $\nu$-{\it symmetric integral} of $g$ with respect to $X$, and satisfies the chain rule
\begin{align*}
f(X_{t})
  &=f(0)+\int_{0}^{t}f^{\prime}(X_{s})d^{\nu}X_{s}.
\end{align*}
The results from \cite{GNRV} provided a method for constructing  Stratonovich-type integrals in the rough-path case where $H < 1/2$. Some well-known examples of measures $\nu$  and their corresponding $\nu$-symmetric Riemann sums are:
\begin{enumerate}
\item  Trapezoidal rule $(\ell = 1)$:  $\nu = \frac 12(\delta_0 + \delta_1)$,
\item  Simpson's rule $(\ell = 2)$:  $\nu = \frac 16(\delta_0 + 4 \delta_{1/2} + \delta_1)$,
\item  Milne's rule $(\ell = 3)$:  $\nu = \frac{1}{90}(7\delta_0 + 32\delta_{1/4}+12\delta_{1/2}+32\delta_{3/4}+7\delta_1)$,
\end{enumerate}
where $\delta_x$ is the Dirac function.  For example, if $\nu = \frac 12 (\delta_0 + \delta_1)$, then \req{nu_integral} is the sum
\[ S_n^\nu(g,t) = \sum_{j=0}^{\Nt -1} \frac{g(X_{\frac jn}) + g(X_{\frac{j+1}{n}})}{2}\DXj, \]
which is the standard Trapezoidal rule from elementary Calculus.  If $X$ is fBm with  Hurst parameter $H > \frac{1}{6}$, then the Trapezoidal rule sum converges in probability as $n$ tends to infinity (see \cite{Cheridito, GNRV}), but in general the limit  does not exist if $H \le \frac{1}{6}$.  

More generally, it is known that $S^\nu_n(g, t)$ does not necessarily converge in probability if $H \le \frac{1}{4\ell+2}$.  Nevertheless, in certain instances of the case $H=\frac{1}{4\ell+2}$, it has been found that $S^\nu_n(g, t)$ converges in law to a random variable with a conditional Gaussian distribution. Cases $\ell =1$ and $\ell=2$ were studied in \cite{NRS} and \cite{Simpson}, respectively.  More recently, Binotto, Nourdin and Nualart have obtained the following general result  for $H = \frac{1}{4\ell+2}$:
\begin{theorem}[\cite{BNN}]\label{thm:BiNoNu}
Assume $X$ is a fBm of Hurst parameter $H = \frac{1}{4\ell+4}$. Consider a function $f\in\Cc^{20\ell+5}(\R)$ such that $f$ and its derivatives  up to the order $20 \ell +5$ have moderate growth (they are bounded by $Ae^{B |x|^\alpha}$, with $\alpha <2$). Then,
\begin{align}
S_{n}^{\nu}(f^{\prime},t)
  \stackrel{\Lc}{\rightarrow}f(X_{t})-f(0)-c_{\nu}\int_{0}^{t}f^{(2\ell+1)}(X_{s})dW_{s}\ \ \ \ \ \text{as}\ n\rightarrow\infty,
\end{align}
where $c_{\nu}$ is some positive constant, $W$ is a Brownian motion independent of $X$ and the convergence holds in the topology of the Skorohod space ${\bf D}[0,\infty)$. 
\end{theorem}
The previous convergence can be written as the following change of variables formula in law:
\begin{align*}
f(X_{t})
  &=f(0)+\int_{0}^{t}f^{\prime}(X_{s})d^{\nu}X_{s}+c_{\nu}\int_{0}^{t}f^{(2\ell+1)}(X_{s})dW_{s}.
\end{align*}

When extending these results to self-similar processes,  surprisingly the critical value is not the scaling  parameter $\beta$ but the {\it increment exponent} $\alpha$ which controls the variance of the increments of $X$ and is defined below.

\begin{definition}
We say that $\alpha$ is the increment exponent for $X$ if for any $0<\epsilon<T<\infty$ there are positive constants $0<c_1 \le c_2$ and $\delta>0$, such that
\begin{align}\label{eq:alphdadef}
c_1 s^\alpha \le \mathbb{E}\left[ (X_{t+s} - X_t)^2\right] \le c_2 s^\alpha,
\end{align}
for every $t\in[\epsilon ,T]$ and $s\in[0,\delta)$.

\end{definition}
The extension of stochastic integration to nonstationary Gaussian processes has been studied  in the papers \cite{Swanson10, HN_12, HN_14}. 
 Each of these papers considered critical values of $\alpha$, for which  particular $\nu$-symmetric Riemann sums $S^\nu_n(g, t)$ converge in distribution (but not necessarily in probability) to a limit which has a Gaussian distribution given  the process $X$. 
 For the fBm, $\alpha =2H$ and the critical value for $\alpha$ coincides with $H=\frac{1}{4\ell+2}$. Papers \cite{Swanson10, HN_14} were both based on the Midpoint integral, and show that the corresponding critical value is $\alpha = \frac{1}{2}$.  Because of the structure of the measure $\nu$, the Midpoint rule integral is not covered in our present paper.  Harnett and Nualart  considered  in \cite{HN_12}  a Trapezoidal integral with $\alpha = \frac{1}{3}$ and  the results in  this paper can be expressed as a special case of Theorem \ref{thm1}  below.

\medskip
\subsection{Main results}\label{sec:mainresult}
Our goal for this paper is to extend the results of \cite{BNN} and \cite{GNRV} to a general  class of self-similar Gaussian processes $X$, and a wider class of functions $g$. In the particular case where $X$ is a  fBm, we extend Theorem \ref{thm:BiNoNu} to the class of functions $f$ with continuous derivatives up to order $8\ell+2$. The idea of the proof is similar to the one presented in \cite{BNN}, but there are technical challenges that arise because in general $X$ is not a stationary process.

Our analysis of the asymptotic distribution of $S_{n}^{\nu}(f^{\prime},t)$ relies heavily on a central limit theorem for the odd variations of $X$, which we establish in Theorem \ref{Variations}. The study of the fluctuations of the variations of $X$ has an interest on its own, and has been extensively studied for the case where $X$ is a fBm (see for instance \cite{RoNuTu} and \cite{CoNuWo}). Nevertheless, Theorem \ref{Variations} is the first one to prove a result of this type for an extended class of self-similar Gaussian process that are  not necessarily stationary.

 For most of the stochastic processes that we consider, such as the fBm and its variants, the self-similarity exponent $\beta$ and the increment exponent $\alpha$ satisfy $\alpha = 2\beta$, but there are examples where $\alpha < 2\beta$. In the sequel, we will assume that the parameters $\alpha$ and $\beta$ satisfy $0<\alpha <1$, $\beta\le 1/2$ and $\alpha \le 2\beta$. Following \cite{Hermite}, we assume as well that the function $\phi$ introduced in (\ref{eq:phicov}), satisfies the following conditions:
\begin{enumerate}[(H.1)]
\item   $\phi$ is twice continuously differentiable in $(1, \infty)$ and for some $\lambda>0$ and $\alpha \in (0,1)$, the function
\begin{equation}
\psi(x)= \phi(x)+ \lambda(x-1) ^\alpha  \label{phi1}
\end{equation}
has a bounded derivative in $(1,2]$.
\item  There are constants $C_1, C_2>0$ and $1< \nu \le 2$ such that 
\begin{equation} \label{ec1}
|\phi''(x)| \le C_1\Indi{(1,2]}(x) (x-1)^{\alpha-2} + C_2\Indi{(2,\infty)}(x) x^{-\nu-1}.
\end{equation}
\end{enumerate}
Although the formulation is slightly different, these hypotheses are equivalent to conditions (H.1) and (H.2) in \cite{Hermite}, with the restrictions
 $\alpha < 1$ and $2\beta \le 1$. In particular, they imply that
\begin{equation} \label{ec2}
|\phi'(x)| \le C'_1\Indi{(1,2]}(x) (x-1)^{\alpha-1} + C'_2\Indi{(2,\infty)}(x) x^{-\nu},
\end{equation}
for some constants $C_1'$ and $C_2'$.
Notice that by Lemma \ref{lem3.2}  in the Appendix,   Hypothesis (H.1) implies that  $\alpha$ is the  increment exponent  of $X$.
Moreover the upper bound  in (\ref{eq:alphdadef}) holds for  any $t\in [0,T]$.  

The following are examples of self-similar processes satisfying the above hypotheses (see  \cite{Hermite}):
\begin{enumerate}
\item[(i)]
{\em  Fractional Brownian motion}. This is a centered Gaussian process with covariance function  given by (\ref{cov}). Here (H.1) and (H.2) hold if $H<\frac{1}{2}$. In this case, 
\[
\phi(x)
  =\frac{1}{2}(1+x^{2H}-(x-1)^{2H}),
\]
 $\alpha =2\beta=2H$ and $\nu=2-2H$.

\item[(ii)]
{\em  Bifractional Brownian motion}.  This is a generalization of the fBm, with covariance given by
\[ R(s,t) = \frac{1}{2^K}\left( (t^{2H} + s^{2H})^K - |t-s|^{2HK}\right)\]
for constants $H\in (0,1)$ and $K \in (0,1]$.  See \cite{Houdre, Lei, RussoTudor06} for properties, and note that $K=1$ gives the classic fBm case. Here (H.1) and (H.2) hold if $HK<1$. For this process we have
\[ \phi(x) = \frac{1}{2^K}\left( (1+x^{2H})^K - (x-1)^{2HK}\right)\]
with $\lambda = 2^{-K}$, $\alpha = 2\beta = 2HK$ and $\nu=(2+2H-2HK)\wedge(3-2HK)-1$.

\item[(iii)]
{\em  Subfractional Brownian motion}.  This Gaussian process has been studied in \cite{Bojdecki, Chavez} and  it has a  covariance given by
\[ R(s,t) = s^{2H} + t^{2H} - \frac 12 \left( (s+t)^{2H} +|s-t|^{2H}\right),\]
with parameter $H\in(0,1)$. Here (H.1) and (H.2) hold if $H<\frac{1}{2}$, in which case $\lambda = 1/2$, $\alpha = 2\beta = 2H$, and
\[ \phi(x) = 1+x^{2H} - \frac 12 \left((x+1)^{2H} + (x-1)^{2H}\right).\]

\item[(iv)]
{\em Two processes in a recent paper by Durieu and Wang}.
For $0 < \alpha < 1$, we consider the centered Gaussian processes $Z_1(t),~Z_2(t)$, with covariances given by:
\begin{align*}
\mathbb{E}\left[ Z_1(s)Z_1(t)\right] &= \Gamma(1-\alpha)\left( (s+t)^\alpha - \max(s,t)^\alpha\right)\\
\mathbb{E}\left[ Z_2(s)Z_2(t)\right] &= \Gamma(1-\alpha)\left( s^\alpha+t^\alpha - (s+t)^\alpha\right),
\end{align*}
where $\Gamma(y)$ denotes the Gamma function. These processes are discussed in a recent paper by Durieu and Wang \cite{DuWa}, where it is shown that the process $Z = Z_1 + Z_2$ (where $Z_1,~Z_2$ are independent) is the limit in law of a discrete process studied by Karlin.  The  process $Z_2$, with a different scaling constant, was first described in Lei and Nualart \cite{Lei}. The corresponding functions $\phi$ of these self-similar processes are:
 \[
 \phi_1(x) =-\Gamma(1-\alpha) (x-1)^{\alpha} +\Gamma(1-\alpha)\left((x-1)^{\alpha}+(x+1)^\alpha - x^\alpha\right)
 \]
 and
\begin{align*} \phi_2(x) &= \Gamma(1-\alpha)(1 + x^\alpha - (x+1)^\alpha)\\
&=-\Gamma(1-\alpha)(x-1)^\alpha + \Gamma(1-\alpha)\left( 1+x^\alpha +(x-1)^\alpha - (x+1)^\alpha\right).
\end{align*}
It is shown in \cite{Hermite} that both $\phi_1$ and $\phi_2$ satisfy (H.1) and (H.2), with $2\beta=\alpha$ and $\nu=2-\alpha$.

\item[(v)]
{\em  Gaussian process in a paper by  Swanson}.  This process was introduced in \cite{Swanson07}, and arises as the limit of normalized empirical quantiles of a system of independent Brownian motions.  The covariance is given by
\[ R(s,t) = \sqrt{st} \sin^{-1}\left(\frac{s\wedge t}{\sqrt{st}}\right),\]
and the corresponding function $\phi$ is given by
\[ \phi(x) = \sqrt{x} \sin^{-1}\left(\frac{1}{\sqrt x}\right).\]
This process has $\alpha = \beta=1/2$ and $\nu=2$, so is an example of the case $\alpha < 2\beta$.\\
\end{enumerate}

It is interesting to remark the differences on the asymptotic behavior of both the power variations and the $\nu$-symmetric integrals of $X$, depending on  whether $\alpha=2\beta$ or $\alpha<2\beta$. As we show in Theorem \ref{Variations}, the process of variations of $X$ satisfies an asymptotic nonstationarity property when $\alpha<2\beta$, which differs from the case $\alpha=2\beta$, where the limit process is a scalar multiple of a Brownian motion.  To better describe this phenomena,  we denote by $Y=\{Y_{t}\}_{t\geq0}$ a continuous centered Gaussian process independent of $X$, with covariance function 
\begin{align}\label{eq:Sigmadef}
\E\left[Y_{s}Y_{t}\right]
  &=\Sigma(s,t):=(t\wedge s)^{\frac{2\beta}{\alpha}},
\end{align}
defined on an enlarged probability space $(\Omega,\mathcal{G},  \Pb)$. The process $Y$ is characterized by the property of independent increments, and 
\begin{align*}
\E\left[(Y_{t+s}-Y_{s})^2\right]
  &=t^{\frac{2\beta}{\alpha}}-s^{\frac{2\beta}{\alpha}}\ \ \ \text{ for }\ \ \ 0\leq s\leq t.
\end{align*}
Notice that for $\alpha<2\beta$, the increments of $Y$ are not stationary and when $\alpha=2\beta$, $Y$ is a standard Brownian motion.  We need the following definition of  stable convergence.
\begin{definition}\label{def:stable}
Assume $\xi_n$ is a sequence  random variables defined on $(\Omega, {\cal F}, \mathbb{P})$ with values on a complete and separable metric space $S$ and $\xi$ is an $S$-valued  random variable defined on the enlarged probability space $(\Omega, {\cal G}, \mathbb{P})$.  We say that $\xi_n$ {\em converges stably} to $\xi$ as $n \to \infty$, if for any continuous and bounded function $f:S \to {\mathbb R}$ and any ${\mathbb R}$-valued, ${\cal F}$-measurable bounded random variable $M$, we have
$$\lim_{n\to \infty} {\mathbb E} \left[f(\xi_n) M\right] = {\mathbb E}\left[f(\xi) M\right].$$
\end{definition}

Next we present a central limit theorem for the odd power variations of $X$, which is a key ingredient for proving Theorem \ref{thm1} and illustrates the asymptotic nonstationarity property that we mentioned before.

\begin{theorem}\label{Variations}  Fix an integer $\ell \ge 1$.  Define the functional
\begin{align}\label{eq:Vndef}
V_n(t) = \sum_{j=0}^{\Nt -1} \DXj^{2\ell +1}, \quad t\ge 0.
\end{align}
If $\alpha = \frac{1}{2\ell +1}$ and the process $X$ satisfies (H.1) and (H.2), then for every $0\leq t_{1},\dots, t_{m}<\infty$, $m\geq1$, the vector $(V_{n}(t_{1}),\dots, V_{n}(t_{m}))$ converges  stably to $\sigma_{\ell}(Y_{t_{1}},\dots, Y_{t_{m}})$, where  
\beq{sigma_ell} 
\sigma_{\ell}^2 = \frac{\alpha}{2\beta}\sum_{r=0}^{\ell -1} K_{r,\ell}\sum_{p \in {\mathbb Z}}\left(|p+1|^\alpha + |p-1|^\alpha-2|p|^\alpha \right)^{2(\ell-r)+1},
\eeq
and $K_{r,\ell} = c_{r,\ell}^2 2^{2r}\lambda^{2\ell+1} (2(\ell-r)+1)!$, where $\lambda$ is the constant appearing in Hypothesis (H.1) and $c_{r,\ell}$ are the coefficients introduced in (\ref{eq:hermite_expansion}).
\end{theorem}
Our main results are Theorems \ref{thm1} and \ref{Cor1} below. 

\begin{theorem}\label{thm1}
Assume  $f\in {\cal{C}}^{8\ell+2}(\R).$  For a given symmetric probability measure $\nu$ and associated integer $\ell(\nu)$, assume the process $X$ satisfies (H.1) and (H.2) with $2\beta\geq \alpha= \frac{1}{2\ell +1}$.  Then, as $n$ tends to infinity,
\[
 \{S_n^\nu(f',t)\}_{t\geq0} \stackrel{Stably}{\longrightarrow} \{f(X_t) - f(0) - \kappa_{\nu,\ell}\sigma_\ell\int_0^t f^{(2\ell +1)}(X_s)~dY_s\}_{t\geq0},
 \]
where $\sigma_\ell$ and $\kappa_{\nu,\ell}$ are the constants given by \req{sigma_ell} and \eqref{def:kappa}, respectively,  and the convergence is in the Skorohod space $\mathbf{D}[0,\infty)$.  Consequently, we have the It\^o-like formula in law
\[
f(X_t) \stackrel{\cal L}{=} f(0) + \int_0^t f'(X_s)~d^\nu X_s + \kappa_{\nu,\ell}\sigma_\ell\int_0^t f^{(2\ell+1)}(X_s)~dY_s.\]
\end{theorem}
The proof of Theorem \ref{thm1} follows the same path as the proof of Theorem 1.1 of Binotto, Nourdin and Nualart \cite{BNN}, but there are technical challenges that arise because in general $X$ is not stationary. The next generalization of the result in \cite{GNRV} easily follows from the proof of Theorem \ref{thm1}.
\begin{theorem}\label{Cor1}
Under the assumptions of Theorem \ref{thm1}, if $\alpha >\frac{1}{2\ell +1}$, then the $\nu$-symmetric  integral  $\int_0^t f'(X_s)~d^\nu X_s$ exists as the limit in probability of the $\nu$-symmetric Riemann sums $S_n^\nu(f',t)$ and for all $t\ge 0$,  we have
\[
f(X_t) = f(0)+\int_0^t f'(X_s)~d^\nu X_s.
\]
\end{theorem}

\noindent The important new developments compared to previous work are:
\begin{itemize}
\item  A system for constructing  stochastic integrals with respect to rough-path processes, originally developed in \cite{BNN, GNRV, Simpson, NRS} for the fBm, is now extended to a wider class of processes that are not necessarily stationary.

\item We prove a central limit theorem for the power variations of general self-similar Gaussian processes.

\item  We present a more efficient proof of tightness, which allows for less restrictions on the integrand function $f$ compared with \cite{BNN}. 
\end{itemize}

The paper is organized as follows. In Section \ref{sec:notation_and_theory} we present some Malliavin calculus preliminaries. In Section \ref{eq:convergence_variations} we prove the convergence of the variations of the process $X.$ Section \ref{proof_main} is devoted to the proofs of Theorems \ref{thm1} and Theorem \ref{Cor1}. Finally, in Section \ref{Sec:App} we prove some technical lemmas.

\medskip
\section{Preliminaries}\label{sec:notation_and_theory}
In the sequel, $X$ will denote a self-similar Gaussian process of self-similarity exponent $\beta$,  satisfying assumptions (H.1) and (H.2), defined on a probability space $(\Omega,\Fc,\Pb)$, where $\Fc$ is the $\sigma$-algebra generated by $X$. This implies that $X$ has the  increment exponent $\alpha$ and we assume  $0<\alpha<1$, $0<\alpha\leq 2\beta$ and $\beta \le  1/2$. Let $Y$ be the continuous Gaussian process with covariance function \eqref{eq:Sigmadef} introduced in Section \ref{sec:mainresult} and  let $Y^{i}$,  $i\geq 1$, be independent copies of $Y$. We will assume that $Y$ and $Y^{i}$, $i\geq 1$, are defined on an enlarged probability space $(\Omega,\Gc,\Pb)$, with $\Fc\subset\Gc$, and they are independent of $X$.

\subsection{Elements of Malliavin Calculus}\label{Notation}
Following are descriptions of some of the identities and methods to be used in this paper. The reader should refer to the texts \cite{Nualart} or \cite{NoP11} for details.
We will denote by $\hten$ the Hilbert space obtained by taking the completion of the space of step functions endowed with the inner product 
\begin{align*}
\Ip{\Indi{[0,a]},\Indi{[0,b]}}_{\hten}
  &:=  \mathbb{E}\left[X_{a}X_{b}\right],
\end{align*}
$a,b \ge 0$, where $\Indi{B}$ is the indicator function of a Borel set $B$.  The mapping $\Indi{[0,t]} \mapsto X_t$ can be extended to a linear isometry between $\hten$ and a Gaussian subspace of $L^2(\Omega, \cal{F}, \mathbb{P})$. We will denote by $X(h)$ the image of $h\in \hten$ by this isometry.
For any integer $q\in\mathbb N$, we denote by $\hten^{\otimes q}$ and $\hten^{\odot q}$ the $q$th tensor product of $\hten$, and the $q$th symmetric tensor product of $\hten$, respectively. The $q$th Wiener chaos	of $L^{2}(\Omega,\cal{F},\mathbb{P})$, denoted by $\mathcal{H}_{q}$, is the closed subspace of $L^{2}(\Omega,\cal{F},\mathbb{P})$ generated by the  random variables $\{H_{q}(X(h)),~  h\in\hten,\Norm{h}_{\hten}=1\}$, where $H_{q}$ is the $q$th Hermite polynomal, defined by 
\begin{align}\label{eq:xnHermite}
H_{q}(x)
  &:=(-1)^{q}e^{\frac{x^{2}}{2}}\frac{\text{d}^{q}}{\text{d}x^{q}}e^{-\frac{x^{2}}{2}}.
\end{align}
We observe that any monomial of the form $x^{2\ell+1}$, for $\ell\in\N$, can be expressed as a linear combination of odd Hermite polynomials with integer coefficients $c_{j,r}$, namely, 
\begin{align}\label{eq:hermite_expansion}
x^{2r+1}
   &=\sum_{j=0}^{r}c_{j,r}H_{2(r-j)+1}(x).
\end{align}
We will denote by $J_{q}$ the projection over the space $\Hc_{q}$. The mapping 
\beq{hmap} I_{q}(h^{\otimes q}):=H_{q}(X(h)),
\eeq 
defined first for $h\in\hten$ such that $\Norm{h}_{\hten}=1$ and then extended by linearity, provides a linear isometry between $\hten^{\odot q}$ (equipped with the norm $\sqrt{q!}\Norm{\cdot}_{\hten^{\otimes q}}$) and $\mathcal{H}_{q}$ (equipped with the $L^{2}$-norm).  The random variable $I_q(\cdot)$ denotes the generalized Wiener-It\^o stochastic integral.

Let $\{e_{n}\}_{n\geq1}$ be a complete orthonormal system in $\hten$. Given $f\in\hten^{\odot p}$, $g\in\hten^{\odot q}$ and $r\in\{0,\dots, p\wedge q\}$, the $r$th-order contraction of $f$ and $g$ is the element of $\hten^{\otimes (p+q-2r)}$ defined by 
\begin{align*}
f\otimes_{r}g
  &=\sum_{i_{1},\dots, i_{r}=1}^{\infty}\Ip{f,e_{i_{1}}\otimes\cdots\otimes e_{i_{r}}}_{\hten^{\otimes r}}\otimes\Ip{g,e_{i_{1}}\otimes\cdots\otimes e_{i_{r}}}_{\hten^{\otimes r}},
\end{align*}
where $f\otimes_{0}g=f\otimes g$, and for $p=q$, $f\otimes_{q}g=\Ip{f,g}_{\hten^{\otimes q}}$. Let $\mathcal{S}$ denote the set  of all cylindrical random variables of the form
\begin{align*}
F= g(X(h_{1}),\dots, X(h_{n})),
\end{align*} 
where $g:\mathbb{R}^{n}\rightarrow\mathbb{R}$ is an infinitely differentiable function with compact support, and $h_{j}\in\hten$. The Malliavin derivative of $F$ is the element of $L^{2}(\Omega;\hten)$, defined by 
\begin{align*}
DF
  &=\sum_{i=1}^{n}\frac{\partial g}{\partial x_{i}}(X(h_{1}),\dots, X(h_{n}))h_{i}.
\end{align*}
By iteration, one can define the $r$th derivative $D^{r}$ for every $r\geq2$, which is an element of $L^{2}(\Omega;\hten^{\odot r})$. For $p\geq1$ and $r\geq1$, the set $\mathbb{D}^{r,p}$ denotes the closure of $\mathcal{S}$ with respect to the norm $\Norm{\cdot}_{\mathbb{D}^{r,p}}$, defined by
\begin{align*}
\Norm{F}_{\mathbb{D}^{r,p}}
  &:=\left(\mathbb{E}\left[\Abs{F}^{p}\right]+\sum_{i=1}^{r}\mathbb{E}\left[\Norm{D^{i}F}_{\hten^{\otimes i}}^{p}\right]\right)^{\frac{1}{p}}.
\end{align*}
The operator $D^{r}$ can be consistently extended to the space $\mathbb{D}^{r,p}$. We denote by $\delta$ the adjoint of the operator $D$, also called the divergence operator. A random element $u\in L^{2}(\Omega;\hten)$ belongs to the domain of $\delta$ in $L^2(\Omega)$, denoted by $\mathrm{Dom} \, \delta $, if and only if satisfies
\begin{align*}
\Abs{\mathbb{E}\left[\Ip{DF,u}_{\hten}\right]}
  &\leq C_{u}\mathbb{E}\left[F^{2}\right]^{\frac{1}{2}},\ \text{ for every } F\in\mathbb{D}^{1,2},
\end{align*}
where $C_{u}$ is  a constant only depending on $u$. If $u\in \mathrm{Dom} \,\delta$, then the random variable $\delta(u)$ is defined by the duality relationship
\begin{align*}
\mathbb{E}\left[F\delta(u)\right]=\mathbb{E}\left[\Ip{DF,u}_{\hten}\right],
\end{align*}
which holds for every $F\in\mathbb{D}^{1,2}$. The previous relation extends to the multiple Skorohod integral $\delta^{q}$, and we have 
\begin{align*}
\mathbb{E}\left[F\delta^{q}(u)\right]
  &=\mathbb{E}\left[\Ip{D^{q}F,u}_{\hten^{\otimes q}}\right],
\end{align*}
for any element $u$ in the domain of $\delta^{q}$, denoted by ${\rm Dom} \, \delta^q$, and any random variable $F\in\mathbb{D}^{q,2}$. Moreover, $\delta^{q}(h)=I_{q}(h)$ for every $h\in\hten^{\odot q}$. For any Hilbert space $V$, we denote by $\mathbb{D}^{k,p}(V)$ the corresponding Sobolev space of $V$-valued random variables (see \cite[page 31]{Nualart}). The operator $\delta^{q}$ is continuous from $\mathbb{D}^{k,p}(\hten^{\otimes q})$ to $\mathbb{D}^{k-q,p}$, for any $p>1$ and any integers $k\geq q\geq1$, that is, we have 
\begin{align}\label{eq:Meyer}
\Norm{\delta^{q}(u)}_{\mathbb{D}^{k-q,p}}
  &\leq c_{k,p}\Norm{u}_{\mathbb{D}^{k,p}(\hten^{\otimes q})}
\end{align}
for all $u\in\mathbb{D}^{k,p}(\hten^{\otimes q})$, and some constant $c_{k,p}>0$. These estimates are consequences of Meyer inequalities (see \cite[Proposition 1.5.7]{Nualart}). In particular, these estimates imply that $\mathbb{D}^{q,2}(\hten^{\otimes q})\subset \mathrm{Dom}~\delta^{q}$ for any integer $q\geq1$. %We have as well, the following  commutativity relationship: $D\delta u-\delta Du=u$ for any $u\in \mathbb{D}^{2,2}(\hten)$.
The following lemma has been proved in \cite[Lemma 2.1]{NoNu}:
\begin{lemma}\label{lem:Fdelta}
Let $q\geq1$ be an integer. Suppose that $F\in\mathbb{D}^{q,2}$, and let $u$ be a symmetric element in ${\rm Dom} \,\delta^{q}$. Assume that, for any $0\leq r+j\leq q$, $\Ip{D^{r}F,\delta^{j}(u)}_{\hten^{\otimes r}}\in L^{2}(\Omega;\hten^{\otimes q-r-j})$. Then, for any $r=0,\dots, q-1$, $\Ip{D^{r}F,u}_{\hten^{\otimes r}}$ belongs to the domain of $\delta^{q-r}$ and we have 
\begin{align*} 
F\delta^{q}(u)
  &:=\sum_{r=0}^{q}\binom{q}{r}\delta^{q-r}(\Ip{D^{r}F,u}_{\hten^{\otimes r}}).
\end{align*}
\end{lemma}
%\noindent The operator $L$ is defined on the Wiener chaos expansion as 
%\begin{align*}
%L=\sum_{q=0}^{\infty}-qJ_{q},
%\end{align*}
%and is the infinitesimal generator of the Ornstein-Uhlenbeck semigroup. The domain of this operator in $L^{2}(\Omega)$ is the set 
%\begin{align*}
%\text{Dom}L
  %&=\{F\in L^{2}(\Omega) : \sum_{q=1}^{\infty}q^{2}\Norm{J_{q}F}_{L^{2}(\Omega)}^{2}<\infty\}.
%\end{align*}

\subsection{Central limit theorems for multiple stochastic integrals}
In the seminal paper \cite{NuPe}, Nualart and Peccati established a central limit theorem for sequences of multiple stochastic integrals of fixed order. In this context, assuming that the variances converge, convergence in distribution to a Gaussian law is actually equivalent to the convergence of just the fourth moment. Shortly afterwards, in \cite{PeTu}, Peccati and Tudor proved a multidimensional version of this result. More recent developments of these type of results have been addressed by using Stein's method and Malliavin calculus (see the monograph by Nourdin and Peccati \cite{NoP11}).

The following result is a version of the multidimensional limit central theorem for multiple stochastic integrals, obtained by Peccati and Tudor in \cite{PeTu}.

\begin{theorem}[Peccati-Tudor criterion]\label{thm:PeTu}
Let $d\geq 2$ and let $q_{1},\dots, q_{d}\geq1$ be some fixed integers. Consider vectors
\begin{align*}
F_{n}=(F_{n}^{1},\dots, F_{n}^{d})=(I_{q_{1}}(h_{n}^{1}),\dots, I_{q_{d}}(h_{n}^{d})),\ \ \ \ \ n\geq1,
\end{align*}
with $h_{n}^{i}\in\Hg^{\otimes q_{i}}$. Let $C=\{C_{i,j}\}_{1\leq i,j\leq d}$ be a $d\times d$ symmetric non-negative definite matrix. Assume that, as $n\rightarrow\infty$, the following conditions hold:
\begin{enumerate}
\item[(i)] For every $i,j=1,\dots, d$,
\begin{align*}
\lim_{n\rightarrow\infty}\E\left[F_{n}^{i}F_{n}^{j}\right]
  &=C_{i,j}.
\end{align*}
\item[(ii)] For all $1\leq i\leq d$ such that $q_{i}>1$ and $1\leq \ell\leq q_{i}-1$, 
$$\Norm{h_{n}^{i}\otimes_{\ell}h_{n}^{i}}_{\Hg^{\otimes 2(q_{i}-\ell)}}\rightarrow0,\ \ \ \text{as }\ \ n\rightarrow\infty.$$
\end{enumerate}
Then $F_{n}$ converges in law to a centered Gaussian law with covariance $C$.
\end{theorem}
We will need the following modification of the Peccati-Tudor criterion, in which we will make use of the notation introduced in Sections \ref{Sec:intro} and \ref{sec:notation_and_theory}.
 
\begin{theorem}\label{thm:NNCLT}
Let $1<q_1<q_{2}<\dots< q_d$ be positive integers.  Consider a sequence of stochastic processes $F_n^{i}=\{F_{n}^{i}(t)\}_{t\geq0}$ of the form $F_n^{i}(t) = I_{q_i} (h_n^i(t))$, where each $h_n^i(t)$ is an element of ${\Hg}^{\otimes {q_i}}$ and $1\leq i\leq d$. Suppose in addition,	that the following conditions hold for every $t\geq0$ and $1\leq i\leq d$:
\begin{enumerate}
\item[(i)]  There exist $c_{1},\dots, c_{d}>0$, such that  for every $s,t\ge 0$ 
\begin{align}\label{conv:hninner}
\lim_{n\rightarrow\infty}\Ip{h_n^i(s),h_n^i(t)}_{{\Hg}^{\otimes {q_i}}}= \frac{c_{i}^2}{q_{i}!}\Sigma(s,t).
\end{align}
\item[(ii)]  For all $i=1,\dots, d$ and $r=1,\dots, q_{i}-1$,
\begin{align}\label{conv:contractionh}
\lim_{n\rightarrow\infty}\Norm{h_n^i(t)\otimes_{r}h_n^i(t)}_{\Hg^{\otimes 2(q_{i}-r)}}=0.
\end{align}
\end{enumerate}
Then the finite dimensional distributions of the process $\sum_{i=1}^{d}F_n^{i}$ converge stably to those of $\sum_{i=1}^{d} \sqrt{q_i!} c_{i}Y^{i}$.
\end{theorem} 
%---------------------
\begin{proof}
Let $t_{1},\dots, t_{m}\ge 0$ be fixed and consider the sequence of random vectors
$$F_{n}=(F_{n}^{i}(t_{j});\ 1\leq i\leq d,\ 1\leq j\leq m).$$
It suffices to show that for any $f_{1},\dots, f_{N}\in\Hg$, the random vectors 
$$\xi_{n}=(F_{n},X(f_{1}),\dots, X(f_{N}))$$ 
converge in  distribution to the Gaussian vector 
$$\xi =(F,X(f_{1}),\dots X(f_{N})),$$
where $X(h)$, for $h\in\Hg$, is defined as in Section \ref{Notation}, $F$ is given by
\begin{align*}
F
  &=(\sqrt{q_i!}c_{i}Y^{i}(t_{j});\ 1\leq i\leq d,\ 1\leq j\leq m),
\end{align*}
and $Y^{1},\dots, Y^{d}$ are the Gaussian processes defined in Section \ref{sec:mainresult}. Notice that in particular $X(f_{i})$ belongs to the chaos of order 1. 

To prove the result we will verify the conditions of Theorem \ref{thm:PeTu}, for $C$ defined as the covariance matrix of $\xi$. Condition \textit{(ii)} follows from \eqref{conv:contractionh}. To prove the convergence of the covariances we proceed as follows. By \eqref{conv:hninner}, the covariance matrix of $F_{n}$ converges to the covariance of $(\sqrt{q_i}c_{i}Y^{i}(t_{j});\ 1\leq i\leq d,\ 1\leq j\leq m)$. In addition, since $q_{i}>1$ for all $1\leq i\leq d$, we have that for every $1\leq  j\leq m$ and $1\leq l\leq N$,
\begin{align*}
\lim_{n\rightarrow\infty}\E\left[F_{n}^{i}(t_{j})X(f_{l})\right]
  &=\E\left[Y^{i}(t_{j})X(f_{l})\right]=0.
\end{align*}
From here, it follows that the covariance of $\xi_{n}$ converges to the covariance of $\xi$, as required. The proof is now complete.
\end{proof}

%======================================================================================================================================

\subsection{Notation}\label{Notationextra}
 For $n \ge 2$ we consider the discretization of $[0,\infty)$ by the points $\{ \frac{j}{n}, j\ge 0\}$.  For $t  \ge 0$, $j\ge 0$ and $n\ge 2$,  we define:
 \[
\varepsilon_t = \Indi{[0,t)},\quad \widetilde{\varepsilon}_{\frac jn} = \frac 12\left(\varepsilon_{\frac jn} + \varepsilon_{\frac{j+1}{n}}\right) \;\text{and }\;\partial_{\frac jn} = \Indi{[\frac jn , \frac{j+1}{n})}.
\]
For the process $X$, we  introduce the notation:
\[ \Delta X_{\frac{t}{n}} = X_{\frac{t+1}{n}}-X_{\frac tn};\quad \widetilde{X}_{\frac tn} = \frac 12\left( X_{\frac{t+1}{n}}+X_{\frac tn}\right)\;\text{ and }\; \xi_{t,n} = \| \Delta X_{\frac{t}{n}}\|_{L^2(\Omega)} .
\]

When not otherwise defined, the symbol $C$ denotes a generic positive constant, which may change from line to line.  The value of $C$ may depend on  the parameters of the process $X$ and the length of the time interval  $[0,t]$ or $[0,T]$  we are considering.

%======================================================================================================================================
 
\section{Asymptotic behavior of the power variations of $X$}\label{eq:convergence_variations}
This section is devoted to the proof of Theorem \ref{Variations}.  
Define $V_{n}(t)$ by \eqref{eq:Vndef} and recall that $\alpha =\frac 1{2\ell +1}$. By the Hermite polynomial expansion of $x^{2\ell+1}$ (see \req{eq:hermite_expansion}) and \req{hmap}, we can write
\[ \frac{\DXj^{2\ell +1}}{\xi_{j,n}^{2\ell+1}} = \sum_{r=0}^{\ell} c_{r,\ell}~H_{2(\ell-r)+1}\left(\frac{\DXj}{\xi_{j,n}}\right) = \sum_{r=0}^\ell c_{r,\ell}~I_{2(\ell-r)+1}\left(\frac{\partial_{\frac jn}^{\otimes 2(\ell-r)+1}}{\xi_{j,n}^{2(\ell-r)+1}}\right),\]
where each $c_{r,\ell}$ is an integer with $c_{0,\ell}=1$. It follows that
\[ \DXj^{2\ell+1} =  \sum_{r=0}^\ell c_{r,\ell}~\xi_{j,n}^{2r} I_{2(\ell-r)+1}\left( \partial_{\frac jn}^{\otimes 2(\ell-r)+1}\right).
\]
Define $q_r = 2(\ell-r)+1$ and notice that $q_{\ell} = 1$ and $3 = q_{\ell -1} < \cdots < q_0 = 2\ell +1$.  We can write for $t \ge 0$
\begin{equation}\label{V_decomp}
V_n(t) = \sum_{r=0}^\ell c_{r,\ell}V_n^r(t),
\end{equation}
where
\[
V_n^r(t) = \sum_{j=0}^{\Nt -1} \xi_{j,n}^{2r} I_{q_r}(\partial_{\frac jn}^{\otimes q_r}) = I_{q_r} (h^{r}_n(t)),
\]
and 
\[ h_n^{r}(t) = \sum_{j=0}^{\Nt -1} \xi^{2r}_{j,n} \partial_{\frac jn}^{\otimes q_r}.
\]
In the next lemma, we show that  the term $V_n^\ell(t)$  does not contribute to the limit of $V_n(t)$ as $n$ tends to infinity.

\begin{lemma}
The term
\[ V_n^\ell(t) =  \sum_{j=0}^{\Nt -1} \xi_{j,n}^{2\ell} I_1(\partial_{\frac jn}) =  \sum_{j=0}^{\Nt -1} \xi_{j,n}^{2\ell} \DXj\]
tends to zero in $L^2(\Omega)$ as $n$ tends to infinity.
\end{lemma}

\begin{proof}
Recalling that $X_0 = 0$ and $\Delta X_{j/n} = X_{(j+1)/n} - X_{j/n}$, we can rewrite the sum as
\[ V_n^\ell(t)=
 X_{\frac{\Nt}{n}}\xi_{\Nt -1,n}^{2\ell} - X_{\frac 1n}\left( \xi_{1,n}^{2\ell} - \xi_{0,n}^{2\ell}\right) - \sum_{j=2}^{\Nt -1} X_{\frac jn}\left( \xi_{j,n}^{2\ell} - \xi_{j-1,n}^{2\ell}\right).
\]
We have, for any integer $j\ge 1$,
\begin{align*}  \xi^2_{j,n} &=\left(\frac{j+1}{n}\right)^{2\beta}\phi(1)+\left(\frac jn\right)^{2\beta}\phi(1) - 2\left(\frac jn\right)^{2\beta}\phi(\frac{j+1}{j})\\
 &= \frac{\phi(1)}{n^{2\beta}}\left( (j+1)^{2\beta} - j^{2\beta}\right) - \frac{2j^{2\beta}}{n^{2\beta}}\left( \phi(1+\frac 1j) - \phi(1)\right).
 \end{align*}
By (H.1), we can write this as
\[
 \xi^2_{j,n} = \frac{2\beta\phi(1)}{n^{2\beta}}\int_0^1 (j+y)^{2\beta-1}dy - \frac{2j^{2\beta}}{n^{2\beta}}\left( -\lambda j^{-\alpha} + \psi(1+\frac 1j) - \psi(1)\right) :=a_n(j).
\]
By the previous formula, we can extend the function $a_n$ to all reals $x\ge 1$.
Using the fact that  $ \psi(x)$ has a bounded derivative in $(1,2]$,
%and taking into account that $\alpha<1$, 
we can find positive constants $C,C^{\prime}$ such that  for all $x\ge 1$,
\[
\left| a_n'(x)\right| \le Cn^{-2\beta}\left( x^{2\beta-2} + x^{2\beta-\alpha -1}\right) \le C^{\prime}n^{-2\beta}x^{2\beta-\alpha -1}.
\]
Hence, by \eqref{ineq:varpartial}, it follows that for integers $2\le j \le \Nt$,
\begin{eqnarray*}
 \left| a_n^\ell (j) - a_n^\ell (j-1)\right| 
&\le&  C \sup_{2\le j \le \Nt}\left|a_n(j)\right|^{\ell -1}\\
&&\times \int_0^1 \left| a_n'(j-1+y)\right|~dy \le Cn^{-(\ell-1)\alpha-2\beta}(j-1)^{2\beta-\alpha-1}.
\end{eqnarray*}
As a consequence, using again inequality (\ref{ineq:varpartial}),  we can write
\begin{eqnarray*}
&&\mathbb{E}\left[\left(X_{\frac{\Nt}{n}}\xi_{\Nt -1,n}^{2\ell} - X_{\frac 1n}\left( \xi_{1,n}^{2\ell} - \xi_{0,n}^{2\ell}\right) - \sum_{j=2}^{\Nt -1} X_{\frac jn}\left( \xi_{j,n}^{2\ell} - \xi_{j-1,n}^{2\ell}\right)\right)^2\right]^{\frac{1}{2}} \\
&& \qquad \le Cn^{-\ell\alpha} + C  \sum_{j=2}^{\Nt -1} \left| a_n^\ell (j) - a_n^\ell (j-1)\right|   \le Cn^{-\ell\alpha},
\end{eqnarray*} 
which tends to zero as $n$ tends to infinity.  
\end{proof}

Then,  Theorem \ref{Variations}  will be  a consequence of Theorem \ref{thm:NNCLT}, if we show that  the  remaining terms $h_n^r(t)$,  $0 \le r \le \ell-1$, $t\geq0$,  satisfy conditions  (\ref{conv:hninner}) and (\ref{conv:contractionh}).  This will be done in the next two lemmas.

\begin{lemma}
Let $1 \le p \le q_r-1$ be an integer. Then,
\[
\lim_{n\rightarrow \infty} \left\| h_n^r(t) \otimes_p h^r_n(t) \right\|^2_{\hten^{\otimes (2q_r - 2p)}} =0.
\]
\end{lemma}

\begin{proof}  We have for each $n \ge 2$
\begin{multline*}
\left\| h_n^r(t) \otimes_p h^r_n(t) \right\|^2_{\hten^{\otimes (2q_r - 2p)}} \\
= \sum_{j_1, j_2, k_1, k_2= 0}^{\Nt -1} \xi_{j_1,n}^{2r}\xi_{j_2,n}^{2r}\xi_{k_1,n}^{2r}\xi_{k_2,n}^{2r}
\left< \partial_{\frac{j_1}{n}}, \partial_{\frac{j_2}{n}}\right>_\hten^p\left< \partial_{\frac{k_1}{n}}, \partial_{\frac{k_2}{n}}\right>_\hten^p\left< \partial_{\frac{j_1}{n}}, \partial_{\frac{k_1}{n}}\right>_\hten^{q_r-p}\left< \partial_{\frac{j_2}{n}}, \partial_{\frac{k_2}{n}}\right>_\hten^{q_r-p}.
\end{multline*}
Note that for applicable values of $q_r$ and $p$ we always have $p\ge 1$ and $q_r-p \ge 1$.  
By (\ref{ineq:varpartial})  and Cauchy-Schwarz inequality, we have
\[
\sup_{0 \le j,k\le \Nt -1} \left| \left< \partial_{\frac{j}{n}}, \partial_{\frac{k}{n}}\right>_\hten\right| \le Cn^{-\alpha}.
\]
As a consequence, 
\begin{eqnarray*}
\left\| h_n^r(t) \otimes_p h^r_n(t) \right\|^2_{\hten^{\otimes (2q_r - 2p)}} & \le & \left( \sup_{0\le j \le \Nt -1} \left| \xi_{j,n}^{2r}\right|\right)^4
\sup_{0 \le j,k\le \Nt -1} \left| \left< \partial_{\frac{j}{n}}, \partial_{\frac{k}{n}}\right>_\hten\right|^{2q_r -3}\\
&&\times \sum_{j_1, j_2, k_1, k_2= 0}^{\Nt -1}\left| \left< \partial_{\frac{j_1}{n}}, \partial_{\frac{j_2}{n}}\right>_\hten\left< \partial_{\frac{j_1}{n}}, \partial_{\frac{k_1}{n}}\right>_\hten\left< \partial_{\frac{j_2}{n}}, \partial_{\frac{k_2}{n}}\right>_\hten\right| \\
 & \le&  C n^{-\alpha(4r+2q_r-3)+1} \left( \sup_{0\le j \le \Nt -1} \sum_{k=0}^{\Nt -1} \left|\left< \partial_{\frac{j}{n}}, \partial_{\frac{k}{n}}\right>_\hten\right|\right)^3.
\end{eqnarray*}
We now apply Lemma \ref{lem:sum_est1} and noting that $4r + 2q_r = 4\ell +2 = \frac{2}{\alpha}$, we have up to a constant $C$,
\[
\left\| h_n^r(t) \otimes_p h^r_n(t) \right\|^2_{\hten^{\otimes (2q_r - 2p)}}
\le  Cn^{-1},
\]
which tends to zero as $n$ tends to infinity.  This completes the proof of the lemma.
\end{proof}

In the next lemma we show  that the functions  $h_n^r$, $0 \le r \le \ell-1$, satisfy condition  (\ref{conv:hninner})  of Theorem \ref{thm:NNCLT}, with  some constants $c_r$ to be defined below.

%=================================================
\begin{lemma}\label{sigma_comps}  Under above notation, let $s,t\ge 0$.  Then for each integer $0 \le r \le \ell -1$,
\begin{equation} \label{equa1}
\lim_{n\to\infty} \left< h_n^r(t) , h_n^r(s) \right>_{\hten^{\otimes q_r}}
= \frac{\alpha}{2\beta }(s\wedge t)^{\frac{2\beta}{\alpha}} 2^{2r} \lambda^{2\ell+1}\sum_{m\in\mathbb{Z}} (\rho_\alpha(m))^{q_r},
\end{equation}
where $\rho_\alpha(m)=|m+1|^\alpha+|m-1|^\alpha-2|m|^\alpha$.
\end{lemma}
\begin{proof}
We can easily check that 
\[
\left< h_n^r(t), h_n^r(s)\right>_{\hten^{\otimes q_r}} 
	=\sum_{j=0}^{\Nt -1}\sum_{k=0}^{\floor{ns} -1}G_{n}(j,k),
\]
where the function $G_{n}(j,k)$ is defined by
\begin{align}\label{eq:Psidef}
G_{n}(j,k)
  =\xi_{j,n}^{2r}\xi_{k,n}^{2r}\left< \partial_{\frac jn},\partial_{\frac {k}{n}}\right>_\hten^{q_r}.
\end{align}
Then the convergence  (\ref{equa1}) will be a consequence of the following two facts:

\noindent
(i) For every $0<s_{1}<t_{1}<s_{2}<t_{2}$,  
\begin{equation}\label{conv:inctozero}
\lim_{n\rightarrow\infty}\sum_{j=\floor{ns_{1}}}^{\floor{nt_{1}} -1}\sum_{k=\floor{ns_{2}}}^{\floor{nt_{2}} -1}\Abs{G_{n}(j,k)}
=0.
\end{equation}

\noindent 
(ii) For every $t>0$,
\begin{equation}\label{eq:sigma_comps}
\lim_{n\rightarrow\infty}\sum_{j,k=0}^{\floor{nt} -1}G_{n}(j,k)\\
= \frac{\alpha}{2\beta}t^{\frac{2\beta}{\alpha}}2^{2r} \lambda^{2\ell+1}\sum_{m\in\mathbb{Z}}(\rho_\alpha(m))^{q_r}.
\end{equation}

First we prove \eqref{conv:inctozero}. 
We can assume that  $n \ge 6$, $\floor{ns_{1}} \ge 1$  and  $\floor {nt_1}+2< \floor{ns_2}$, which is true if $n$ is large enough. This implies that $j+3 \le k$ for each $k$ and $j$ such that
$\floor{ns_{1}} \le j \le\floor{nt_{1}} -1$ and  $\floor{ns_{2}} \le k \le\floor{nt_{2}} -1$. As a consequence, applying  inequalities (\ref{ecua1}) and (\ref{ineq:covpartial1}),  we obtain the estimate
\begin{align*}
\sum_{j=\floor{ns_{1}}}^{\floor{nt_{1}} -1} 
\sum_{ k=\floor{ns_{2}} }^{(2j+2)\wedge(\floor{nt_{2}}-1)}|G_n(j,k)|  &\le C\sum_{j=\floor{ns_{1}}}^{\floor{nt_{1}} -1} \sum_{ k=\floor{ns_{2}}}^{\floor{nt_{2}}-1}
  n^{-4\beta r} k^{(2\beta-\alpha)r} j^{(2\beta-\alpha)r} n^{-2\beta q_r} j^{(2\beta-\alpha)q_r} k^{(\alpha-2)q_r}\\
	&\leq C  n^{2-2(\alpha r +q_r)},
\end{align*}
which converges to  zero as $n$ tends to  infinity due to the fact that $\alpha>0$ and $q_{r}\geq 1.$ On the other hand, applying  inequalities (\ref{ecua1}) and (\ref{ineq:covpartial2})  we obtain the estimate
\begin{align*}
\sum_{j=\floor{ns_{1}}}^{\floor{nt_{1}} -1} 
\sum_{ k=(2j+2)\vee\floor{ns_{2}}}^{\floor{nt_{2}} -1}|G_n(j,k)|  &\le C\sum_{j=\floor{ns_{1}}}^{\floor{nt_{1}} -1} \sum_{ k=\floor{ns_{2}}} ^{\floor{nt_{2}} -1}
  n^{-4\beta r} j^{(2\beta-\alpha)r} k^{(2\beta-\alpha)r} n^{-2\beta q_r} j^{(2\beta+\nu -2)q_r} k^{-\nu q_r} \\
  &\leq Cn^{2-2(\alpha r+q_{r})}.
\end{align*}
The exponent of $n$ is the above estimate is always negative, so this term converges to zero as $n$ tends to infinity.

Next we prove \eqref{eq:sigma_comps}. We can write
\begin{equation}\label{Variance_r}
\sum_{j,k=0}^{\floor{nt} -1}G_{n}(j,k)
  = \sum_{x=0}^{\floor{nt}-1}\sum_{j=0}^{\floor{nt} -1-x}(2-\delta_{x,0})G_{n}(j,j+x),
\end{equation}
where $\delta_{x,0}$ denotes the Kronecker delta.  
First we will show that there exist constants $C,\delta>0$, such that for $3\leq x\leq\floor{nt}-1$, 
\begin{equation} \label{ecua17}
\sum_{j=0}^{\floor{nt} -1-x}(2-\delta_{x,0})|G_{n}(j,j+x)|
  \leq  C x^{-1-\delta}.
\end{equation}
To show (\ref{ecua17}) we consider three cases:

\noindent
Case 1:  For $j=0$, we have, using (\ref{ecua1}) and   (\ref{ec2}),
\begin{eqnarray}\label{eq:ecua17p0}
|G(0,x)|  &\le& C n^{-4\beta r}  x^{(2\beta-\alpha)r} | n^{-2\beta} (\phi(x+1) -\phi(x))|^{q_r} \le C n^{-2\beta(2\ell+1)} x^{ (2\beta-\alpha)r-\nu q_r} \nonumber\\
&\le &  C x^{-2\beta(2\ell+1) +(2\beta-\alpha)r-\nu q_r}
\end{eqnarray}
which provides the desired estimate, because  the largest value of the exponent $-2\beta(2\ell+1) +(2\beta-\alpha)r-\nu q_r$ is obtained for $r=\ell-1$, and in this case this exponent becomes
\[
-2\beta(\ell+2) -(\ell-1) \alpha -3\nu \le -\alpha(2\ell+1) -3\nu= -1-3\nu.
\]

\noindent
Case 2: Applying  (\ref{ineq:covpartial1}), yields
\begin{align*}
\sum_{j=x-2}^{\floor{nt} -1-x} |G_{n}(j,j+x)|
 &\le C\sum_{j=x-2}^{\floor{nt} -1-x}n^{-2\beta(2\ell+1)}j^{(2\beta-\alpha)(r+q_{r})}(j+x)^{(2\beta-\alpha)r+(\alpha-2)q_{r}}\\
&\le C\sum_{j=x-2}^{\floor{nt} -1-x}n^{-2\beta(2\ell+1)}(j+x)^{(2\beta-\alpha)(2\ell+1)+(\alpha-2)q_{r}}\\
&\le C^{\prime}\sum_{j=x-2}^{\floor{nt} -1-x}(j+x)^{-\alpha(2\ell+1)+(\alpha-2)q_{r}}.
\end{align*}
Hence, using the bound 
$
(j+x)^{(\alpha-2)q_{r}}
  \leq j^{(\alpha-2)(q_{r}-1)}x^{\alpha-2},
$
and the condition $\alpha=\frac{1}{2\ell+1}$, we get 
\begin{align}\label{eq:ecua17p1}
\sum_{j=x-2}^{\floor{nt} -1-x} |G_{n}(j,j+x)|
 &\le Cx^{\alpha-2}\sum_{j=1}^{\infty}j^{-1+(\alpha-2)(q_{r}-1)}.
\end{align}
The sum in the right hand side is finite due to the fact that $q_{r}\geq 3$ and $\alpha<1$.\\

\noindent
Case 3: By  (\ref{ineq:covpartial2}), 
\begin{align}\label{eq:Gnc3}
\sum_{j=0}^{x-2}\Abs{G_{n}(j,j+x)}
  &\leq C \sum_{j=0}^{x-2}n^{-2\beta(2\ell+1)}j^{2\beta(r+q_{r})-\alpha r+(\nu-2)q_{r}}(j+x)^{(2\beta-\alpha)r-\nu q_{r}}.
\end{align}
Notice that
\begin{align*}
n^{-2\beta(2\ell+1)}j^{2\beta(r+q_{r})}(j+x)^{2\beta r}
  &\leq n^{-2\beta(2\ell+1)}(j+x)^{2\beta(2\ell+1)}\leq C
\end{align*}
and 
\begin{align*}
(j+x)^{-\nu q_{r}}
  &\leq j^{-\nu(q_{r}-1)}x^{-\nu}.
%j^{(\nu-2)q_{r}}(j+x)^{-\nu q_{r}}\leq x^{-\nu}j^{-2 q_{r}+\nu}.
\end{align*}
Hence, by \eqref{eq:Gnc3}, 
\begin{align}\label{eq:ecua17p2}
\sum_{j=0}^{x-2}\Abs{G_{n}(j,j+x)}
  &\leq \sum_{j=0}^{x-2}j^{-\alpha r+(\nu-2)q_{r}}(j+x)^{-\alpha r-\nu q_{r}}\leq x^{-\nu}\sum_{j=0}^{x-2}j^{-\alpha r-2q_{r}+\nu}(j+x)^{-\alpha r}\nonumber\\
  &\leq C x^{-\nu}\sum_{j=0}^{x-2}j^{-2\alpha r-2 q_{r}+\nu}.
\end{align}
The sum in the right hand side is finite due to the conditions $q_{r}\geq 3$ and $\nu\leq 2$.\\

\noindent Relation \eqref{ecua17} follows from \eqref{eq:ecua17p0}, \eqref{eq:ecua17p1} and \eqref{eq:ecua17p2}. As a consequence, provided that we prove the pointwise convergence 
\begin{equation}\label{conv:Gpointwise}
\lim_{n\rightarrow\infty}\sum_{j=0}^{\floor{nt} -1-x}G_{n}(j,j+x)
=\frac{\alpha}{2\beta}2^{2\ell+1-q_{r}}\lambda^{2\ell+1}t^{\frac{2\beta}{\alpha}}( \rho_\alpha(x))^{q_{r}},
\end{equation}
for any $x\ge 0$, by applying the dominated convergence theorem in \eqref{Variance_r}, we obtain \eqref{eq:sigma_comps}. 
The proof of \eqref{conv:Gpointwise} will be done in three steps.

\medskip
\noindent
{\it Step 1.} 
 Since $\phi(y)=-\lambda (y-1)^{\alpha}+\psi(y)$, for every $x\geq1$ we can write
\begin{multline*}
\E\left[(X_{j+1}-X_{j})(X_{j+x+1}-X_{j+x})\right]\\
\begin{aligned}
  &=(j+1)^{2\beta}(\phi(1+\frac{x}{j+1})-\phi(1+\frac{x-1}{j+1}))+j^{2\beta}(\phi(1+\frac{x}{j})-\phi(1+\frac{x+1}{j}))\\
	&=-\lambda(j+1)^{2\beta-\alpha}(x^{\alpha}-(x-1)^{\alpha})-\lambda j^{2\beta-\alpha}(x^{\alpha}-(x+1)^{\alpha})\\
	&+(j+1)^{2\beta}(\psi(1+\frac{x}{j+1})-\psi(1+\frac{x-1}{j+1}))+j^{2\beta}(\psi(1+\frac{x}{j})-\psi(1+\frac{x+1}{j})).
\end{aligned}
\end{multline*}
Hence, using  the Mean Value Theorem for  $\psi$, as well as (H.1), we deduce that for every $x\geq1$, there exist constants $\gamma_{1}$ and $ \gamma_{2}>0$, such that
\begin{multline*}
\E\left[(X_{j+1}-X_{j})(X_{j+x+1}-X_{j+x})\right]\\
\begin{aligned}
  &=-\lambda(j+1)^{2\beta-\alpha}(x^{\alpha}-(x-1)^{\alpha})-\lambda j^{2\beta-\alpha}(x^{\alpha}-(x+1)^{\alpha})\\
	&+(j+1)^{2\beta-1} \psi'(1+\gamma_1) - j^{2\beta-1} \psi'(1+\gamma_2).
\end{aligned}
\end{multline*}
As a consequence, taking into account that $\psi'$ is bounded and $\alpha <1$,  
\begin{align}\label{conv:covincrements}
\lim_{j\rightarrow\infty}(j+1)^{\alpha-2\beta}\E\left[(X_{j+1}-X_{j})(X_{j+x+1}-X_{j+x})\right]
  &=-\lambda(2x^{\alpha}-(x-1)^{\alpha}-(x+1)^{\alpha}).
\end{align}
In addition, from Lemma \ref{lem3.2}, it follows that
\begin{align}\label{conv:xi}
\lim_{j\rightarrow\infty}(j+1)^{\alpha-2\beta}\E\left[(X_{j+1}-X_{j})^2\right]
=\lim_{j\rightarrow\infty}(j+1)^{\alpha-2\beta}\E\left[(X_{j+x+1}-X_{j+x})^2\right]
=2\lambda.
\end{align}
Using \eqref{conv:covincrements} and \eqref{conv:xi}, we get 
\begin{align}\label{eq:covnormpointw}
\lim_{j\rightarrow\infty}\xi_{j,1}^{-1}\xi_{j+x,1}^{-1}\E\left[(X_{j+1}-X_{j})(X_{j+x+1}-X_{j+x})\right]
	&= \frac{1}{2}(\Abs{x-1}^{\alpha}+\Abs{x+1}^{\alpha}-2\Abs{x}^{\alpha}).
\end{align}
Notice that the previous relation is also true for $x=0$. Therefore,  we deduce that for every $\varepsilon>0$, there exists $M>0$, such that for every $j\geq M$,
\begin{equation} \label{eps}
\bigg|\xi_{j,1}^{-q_{r}}\xi_{j+x,1}^{-q_{r}}\E\left[(X_{j+1}-X_{j})(X_{j+x+1}-X_{j+x}) \right]^{q_r}
-2^{-q_{r}}( \rho_\alpha(x))^{q_{r}}\bigg|<\varepsilon.
\end{equation}

\medskip
\noindent
{\it Step 2.}  Provided that we prove that 
\begin{align}\label{conv:xiaux}
\lim_{n\rightarrow\infty}n^{-\frac{2\beta}{\alpha}}\sum_{j=0}^{\floor{nt}-1-x}\xi_{j,1}^{2\ell+1}\xi_{j+x,1}^{2\ell+1}=\frac{\alpha}{2\beta}(2\lambda)^{2\ell+1}t^{\frac{2\beta}{\alpha}},
\end{align}
taking into account the self-similarity of the process $X$, and the fact that $\alpha =\frac 1{ 2\ell+1}$, the proof of  \eqref{conv:Gpointwise} will follow from
\begin{equation} \label{ecua67}
\lim_{n\rightarrow \infty} 
\sum_{j=0}^{\floor{nt}-1-x}\Abs{\xi_{j,n}^{2\ell+1-q_{r}}\xi_{j+x,n}^{2\ell+1-q_{r}}\left< \partial_{\frac jn},\partial_{\frac {j+x}{n}}\right>_\hten^{q_r}-
2^{-q_{r}} \xi_{j,n}^{2\ell+1}\xi_{j+x,n}^{2\ell+1}( \rho_\alpha(x))^{q_{r}}} =0.
\end{equation}
Using \eqref{ineq:varpartial} we can easily prove that 
\begin{equation}\label{eq:nearvariance2}
\limsup_{n\rightarrow\infty}\sum_{j=0}^{M}\Abs{\xi_{j,n}^{2\ell+1-q_{r}}\xi_{j+x,n}^{2\ell+1-q_{r}}\left< \partial_{\frac jn},\partial_{\frac {j+x}{n}}\right>_\hten^{q_r}-2^{-q_{r}}\xi_{j,n}^{2\ell+1}\xi_{j+x,n}^{2\ell+1}( \rho_\alpha(x))^{q_{r}}}=0.
\end{equation} 
Applying the estimate (\ref{eps}) and the limit (\ref{conv:xiaux}), we obtain
\begin{eqnarray} \notag
 &&\limsup_{n\rightarrow \infty} \sum_{j=M}^{\floor{nt}-1-x}\Abs{\xi_{j,n}^{2\ell+1-q_{r}}\xi_{j+x,n}^{2\ell+1-q_{r}}\left< \partial_{\frac jn},\partial_{\frac {j+x}{n}}\right>_\hten^{q_r}-
2^{-q_{r}} \xi_{j,n}^{2\ell+1}\xi_{j+x,n}^{2\ell+1}( \rho_\alpha(x))^{q_{r}}} \\ \notag
 && \quad =\limsup_{n\rightarrow \infty}  n^{-\frac {2\beta}{\alpha}} \sum_{j=M}^{\floor{nt}-1-x} \xi_{j,1}^{2\ell+1}\xi_{j+x,1}^{2\ell+1} \\ \notag
&&\qquad  \times   \left|\xi_{j,1}^{-q_{r}}\xi_{j+x,1}^{-q_{r}}\E\left[(X_{j+1}-X_{j})(X_{j+x+1}-X_{j+x}) \right]^{q_r}
-2^{-q_{r}}( \rho_\alpha(x))^{q_{r}}\right| \\ 
&& \qquad \le \varepsilon \frac{\alpha}{2\beta}(2\lambda)^{2\ell+1}t^{\frac{2\beta}{\alpha}}.  \label{ecua68}
\end{eqnarray}
Therefore, (\ref{eq:nearvariance2}) and (\ref{ecua68})  imply (\ref{ecua67}).

\medskip
 \noindent
 {\it Step 3.}
In order to prove \eqref{conv:xiaux} we proceed as follows. Using Lemma \ref{lem3.2}, as well as the condition $\alpha<1$, we deduce that for every $\varepsilon>0$, there exists $M\in\N$, such that for every $j\geq M$, 
$$\Abs{(j^{-(2\beta-\alpha)}\xi_{j,1}\xi_{j+x,1})^{2\ell+1}-(2\lambda)^{2\ell+1}}<\varepsilon,$$
and hence, since $\alpha=(2\ell+1)^{-1}$,
\begin{eqnarray*}
&&n^{-\frac{2 \beta}{\alpha}}\sum_{j=M}^{\floor{nt}-1-x}\Abs{\xi_{j,1}^{2\ell+1}\xi_{j+x,1}^{2\ell+1}-(2\lambda)^{2\ell+1}j^{(2\beta-\alpha)(2\ell+1)}}\\
&&\qquad \qquad =n^{-\frac{2\beta}{\alpha}}\sum_{j=M}^{\floor{nt}-1-x}j^{(2\beta-\alpha)(2\ell+1)} \left| \xi_{j,1}^{2\ell+1}\xi_{j+x,1}^{2\ell+1} j^{-(2\beta-\alpha)(2\ell+1)}
-(2\lambda)^{2\ell+1} \right|  \\
 && \qquad \qquad  \leq \varepsilon n^{-\frac{2\beta}{\alpha}}\sum_{j=M}^{\floor{nt}-1-x}j^{\frac{2\beta-\alpha}{\alpha}}.
\end{eqnarray*}
Therefore, since 
\begin{align}\label{conv:j}
\lim_{n\rightarrow\infty}n^{-\frac{2\beta}{\alpha}}\sum_{j=0}^{\floor{nt}-1-x}j^{\frac{2\beta-\alpha}{\alpha}}
  &= \frac{\alpha}{2\beta}t^{\frac{2\beta}{\alpha}},
\end{align}
we conclude that there exists a constant $C>0$ depending on $t$ and $x$, such that
\begin{align*}
\limsup_{n\rightarrow\infty}n^{-\frac{2\beta}{\alpha}}\sum_{j=M}^{\floor{nt}-1-x}\Abs{\xi_{j,1}^{2\ell+1}\xi_{j+x,1}^{2\ell+1}-(2\lambda)^{2\ell+1}j^{((2\beta-\alpha)(2l+1)}}
  &< C\varepsilon,
\end{align*}
and hence, by relation \eqref{conv:j} and condition $\alpha=(2\ell+1)^{-1}$, we conclude that
\begin{multline*}
\lim_{n\rightarrow\infty}n^{-2\beta(2\ell+1)}\sum_{j=0}^{\floor{nt}-1-x}\xi_{j,1}^{2\ell+1}\xi_{j+x,1}^{2\ell+1}\\
  =(2\lambda)^{2\ell+1}\lim_{n\rightarrow\infty}n^{-2\beta(2\ell+1)}\sum_{j=0}^{\floor{nt}-1-x}j^{(2\beta-\alpha)(2\ell+1)}
  =\frac{\alpha}{2\beta} (2\lambda)^{2\ell+1}t^{\frac{2\beta}{\alpha}},
\end{multline*} 
as required. The proof of Lemma \ref{sigma_comps} is now complete. 
\end{proof}
%==============================================================
%\subsection{Estimations on the covariances of the increments of $X$}
%===========================================================================================================
%===========================================================================================================
\section{Asymptotic behavior of $S_{n}^{\nu}(f^{\prime},t)$}\label{proof_main}
In this section we prove the main results, Theorems \ref{thm1} and \ref{Cor1}. We follow arguments similar to those used in  the proof of Theorem 1.1 of Binotto, Nourdin and Nualart \cite{BNN}, which was originally used in \cite{GNRV}.  For $f\in\mathcal{C}^{8\ell+2}(\R)$ and $a<b$, we consider the approximation \eqref{eq:Taylor_expansion} below, which was proved in  \cite[Theorem~3.6]{GNRV}  using Taylor's formula and the properties of $\nu$
\begin{eqnarray}  \notag
f(b)& =&  f(a) + (b-a)\int_0^1 f'(a+y(b-a))~\nu(dy) + \sum_{h=\ell}^{2\ell} \kappa_{\nu,h} f^{(2h+1)}\left(\frac{a+b}{2}\right) (b-a)^{2h+1}\\ 
&& + C(a,b)(b-a)^{4\ell+2}, \label{eq:Taylor_expansion} 
\end{eqnarray}
where $C(a,b)$ is a continuous function with $C(a,a)=0$, and the $\kappa_{\nu,h}$ are the constants given in  \cite[Theorem~3.6]{GNRV}.  In particular,
\begin{align}\label{def:kappa}
 \kappa_{\nu,\ell} =\frac{1}{(2\ell)!}\left( \frac{1}{(2\ell+1)2^{2\ell}} - \int_0^1 \left(y-\frac 12\right)^{2\ell} \nu(dy)\right). 
\end{align}
Recall the notation $\widetilde{X}_{\frac{t}{n}}$ and $\Delta X_{\frac{t}{n}}$  introduced in Section \ref{Notationextra}. From \eqref{eq:Taylor_expansion}, it follows that for $n\ge 2$,
\begin{align}\label{eq:Taylorexp}
f(X_{t})-f(0)
  &=S_{n}^{\nu}(f^{\prime},t)+\sum_{h=\ell}^{2\ell}\sum_{j=0}^{\floor{nt}-1}\kappa_{\nu,h}f^{(2h+1)}(\widetilde{X}_{\frac{j}{n}})(\DXj)^{2h+1}+R_{n}(t),
\end{align} 
where 
\begin{align*}
R_{n}(t)
  &=\sum_{j=0}^{\floor{nt}-1}C(X_{\frac{j}{n}},X_{\frac{j+1}{n}})(\DXj)^{4\ell+2}.
\end{align*}
Then, we can write 	
\begin{align}\label{eq:approxItoformula}
f(X_{t})-f(0)
  &=S_{n}^{\nu}(f^{\prime},t)+\sum_{h=\ell}^{2\ell}\Phi_{n}^{h}(t)+R_{n}(t),
\end{align} 
where
\begin{align}
\Phi_{n}^{h}(t)
  &=\kappa_{\nu,h}\sum_{j=0}^{\floor{nt}-1}f^{(2h+1)}(\widetilde{X}_{\frac{j}{n}})(\DXj)^{2h+1}.\label{def:Phi}
\end{align}
The term $R_{n}$ converges to zero in probability, uniformly in compact sets. Indeed, for every $T,K,\varepsilon>0$, we can write
\begin{align}\label{ineq:Cres1}
\mathbb{P}\left[\sup_{0\leq t\leq T}\Abs{R_{n}(t)}>\varepsilon\right]
  &\leq\mathbb{P}\left[\sup_{\substack{s,t\in[0,T]\\\Abs{t-s}\leq\frac{1}{n}}}\Abs{C(X_{s},X_{t})}>\frac{1}{K}\right]+\mathbb{P}\left[\sum_{j=0}^{\floor{nT}-1}(\DXj)^{4\ell+2}>K\varepsilon\right].
\end{align}
Since $\Delta X_{\frac{j}{n}}$ is a centered Gaussian variable, by \eqref{ineq:varpartial},   for all even integer $r$  
$$\sup_{1\leq j\leq \floor{nT}-1}\mathbb{E}\left[\Abs{\DXj}^r\right]\leq (r-1)!!\sup_{1\leq j\leq \floor{nT}-1}\mathbb{E}\left[\Abs{\DXj}^2\right]^{\frac{r}{2}}\leq C(r-1)!!n^{-\frac{r\alpha}{2}},$$ 
where $(r-1)!!$ denotes the double factorial $(r-1)!!=\prod_{k=0}^{r-1}(r-1-2k)$. As a consequence, using the Chebychev inequality and the condition $\alpha=\frac{1}{2\ell+1}$, we get
\begin{align}\label{ineq:Cres2}
\mathbb{P}\left[\sum_{j=0}^{\floor{nT}-1}(\DXj)^{4\ell+2}>K\varepsilon\right]
  &\leq \frac{C}{K\varepsilon}\frac{\floor{nT}}{n}\leq\frac{C}{K\varepsilon}.
\end{align}
The  convergence to zero in  probability, uniformly in compact sets,  of $R_{n}(t)$ is obtained from \eqref{ineq:Cres1} and \eqref{ineq:Cres2}, by letting first $n\rightarrow\infty$, and then $K\rightarrow\infty$.

The previous analysis shows that the term $R_{n}$ appearing in right hand side \eqref{eq:approxItoformula}, does not contribute to the limit as $n$ goes to infinity, so the asymptotic behavior of $S_{n}^{\nu}(f',t)$ is completely determined by $\sum_{h=\ell}^{2\ell}\Phi_{n}^{h}(t)$. The study of the  stochastic process $\sum_{h=\ell}^{2\ell}\Phi_{n}^{h}$ can be decomposed in the following steps: first, we reduce the problem of proving  Theorems \ref{thm1} and \ref{Cor1}, to the case where $f$ is compactly supported, by means of a localization argument. Then we prove that the processes $\Phi_{n}^{h}(t)$, with $h=\ell,\dots, 2\ell$ are tight in the Skorohod topology, and only contribute to the limit as $n$ goes to infinity, when $h=\ell$. 

Finally, we determine the behavior of $\Phi_{n}^{\ell}$ by splitting into the cases $\alpha=\frac{1}{2\ell+1}$ and $\alpha>\frac{1}{2\ell+1}$. In the case $\alpha>\frac{1}{2\ell+1}$, we show that $\Phi_{n}^{\ell}\rightarrow0$ in probability, which proves Theorem \ref{Cor1}. For the case $\alpha=\frac{1}{2\ell+1}$, we use the small blocks-big blocks methodology (see \cite{BNN} and \cite{CoNuWo}) and Theorem \ref{Variations}, to prove that $\Phi_{n}^{\ell}$ converges stably to $\{\kappa_{\nu,\ell}\sigma_{\ell}\int_{0}^{t}f^{(2\ell+1)}(X_{s})d Y_{s}\}_{t\geq0}$, which proves Theorem \ref{thm1}.

 We start reducing the problem of proving Theorems \ref{thm1} and \ref{Cor1}, to the case where $f$ is compactly supported. 
 Define the process $Z=\{Z_{t}\}_{t\geq0}$, by
\begin{align}\label{def:Z}
Z_{t}
  &=\kappa_{\nu,\ell}\sigma_{\ell}\int_{0}^{t}f^{(2\ell+1)}(X_{s})dY_{s}.
\end{align}
By \eqref{eq:approxItoformula}, it suffices to show that for all $f\in \mathcal{C}^{8\ell+2}(\R)$, the following claims hold:
\begin{enumerate}
\item If $\alpha=\frac{1}{2\ell+1}$, 
\begin{align}\label{eq:tighteq1}
\{\sum_{h=\ell}^{2\ell}\Phi_{n}^{h}(t)\}_{t\geq0}\stackrel{stably}{\rightarrow}\{ Z_t\}_{t\geq0}~~~~~~\text{as}~n\rightarrow\infty,
\end{align}
in the topology of $\textbf{D}[0,\infty)$.
\item If $\alpha>\frac{1}{2\ell+1}$, then for every $t\ge 0$
\begin{align}\label{eq:tighteq2}
\sum_{h=\ell}^{2\ell}\Phi_{n}^{h}(t)\stackrel{\Pb}{\rightarrow}0,~~~~~~\text{as}~n\rightarrow\infty.
\end{align}
\end{enumerate}
 Notice  that the convergences \eqref{eq:tighteq1}  and \eqref{eq:tighteq2} hold, provided that:
\begin{enumerate}
\item If $\alpha=\frac{1}{2\ell+1}$, then,
\begin{enumerate}
\item[a)] For every $h=\ell,\dots, 2\ell$, the sequence $\Phi_{n}^{h}$ is tight in $\textbf{D}[0,\infty)$.
\item[b)] The finite dimensional distributions of $\Phi_{n}^{\ell}$ converge stably to those of $Z$. 
\item[c)] For every $h=\ell+1,\dots,2\ell$ and $t\ge0$, the sequence $\Phi_{n}^{h}(t)$ converges to zero in probability.
\end{enumerate}
\item If $\alpha>\frac{1}{2\ell+1}$, then $\Phi_{n}^{h}(t)$ converges to zero in probability for every $h=\ell,\dots,2\ell$ and $t>0$.
\end{enumerate}
In turn, these conditions are a consequence of the following claims:
\begin{enumerate}[(i)]
\item For every $\varepsilon, T>0$ and $h=\ell,\dots, 2\ell$, there is a compact set $K\subset \textbf{D}[0, T]$, such that 
\begin{align*}
\sup_{n\geq 1}\mathbb{P}\left[\Phi_{n}^{h}\in K^{c}\right]<\varepsilon.
\end{align*}
\item For every $\varepsilon,\delta>0$, $t\geq0$ and $h=\ell+1,\dots, 2\ell$, there exists $N\in\mathbb{N}$, such that for every $n\geq N$, 
\begin{align*}
\mathbb{P}\left[\Abs{\Phi_{n}^{h}(t)}>\delta\right]<\varepsilon.
\end{align*} 
\item Let $\varepsilon>0$ and $0\leq t_{1}\leq\cdots\leq t_{d}\leq T$ be fixed. If $\alpha=\frac{1}{2\ell+1}$, then for every compactly supported function $\phi\in {\cal{C}}^{1}(\mathbb{R}^{d},\mathbb{R})$, and every event $B\in \sigma(X)$, there exists $N\in\mathbb{N}$, such that for $n\geq N$,
\begin{align}\label{conv:fdd}
\Abs{\mathbb{E}\left[(\phi(\Phi_{n}^{\ell}(t_{1}),\dots, \Phi_{n}^{\ell}(t_{d}))-\phi(Z_{t_{1}},\dots, Z_{t_d}))\Indi{B}\right]}
  &<\varepsilon.
\end{align}
\item If $\alpha>\frac{1}{2\ell+1}$, then for every $\varepsilon,\delta>0$, $t\geq0$ there exists $N\in\mathbb{N}$, such that for every $n\geq N$, 
\begin{align*}
\mathbb{P}\left[\Abs{\Phi_{n}^{\ell}(t)}>\delta\right]<\varepsilon.
\end{align*} 
\end{enumerate}

Recall that $\Phi_{n}^{h}$ depends on $f$ via \eqref{def:Phi}. We claim that it suffices to show conditions (i)-(iv) for $f$ compactly supported. Suppose that (i)-(iv) hold for every function in ${\cal{C}}^{8\ell+2}(\mathbb{R})$ with compact support, and take a general element $g\in {\cal{C}}^{8\ell+2}(\mathbb{R})$. Fix $L\ge 1$ and let $g_{L}:\mathbb{R}\rightarrow\mathbb{R}$ be a compactly supported function, with derivatives up to order $8\ell+2$, such that $g_{L}(x)=g(x)$ for every $x\in[-L,L]$, and define the processes $\widetilde{\Phi}_{n}^{h,L}=\{\widetilde{\Phi}_{n}^{h,L}(t)\}_{t\geq0}$, $h=\ell,\dots, 2\ell$ and $\widetilde{Z}^{L}=\{\widetilde{Z}_{t}^{L}\}_{t\geq0}$,  by
\begin{align*}
\widetilde{\Phi}_{n}^{h,L}
  &=k_{\nu,h}\sum_{j=0}^{\floor{nt}-1}g_{L}^{(2h+1)}(\widetilde{X}_{\frac{j}{n}})(\DXj)^{2h+1},
\end{align*} 
and
\begin{align*}
\widetilde{Z}_{t}^{L}
  &=\kappa_{\nu,\ell}\sigma_{\ell}\int_{0}^{t}g_{L}^{(2\ell+1)}(X_{s})dY_{s}.
\end{align*}

Fix $T>0$ and define as well the events $A_{L,T}=\{\sup_{0\leq s\leq T}\Abs{X_{s}}\leq L\}$. Then, for every  $\varepsilon>0$,  there exists a compact set $K_{L}\subset \textbf{D}[0,T]$  such that  for all $h=\ell,\dots,2\ell$
\begin{align}
\sup_{n\geq1}\mathbb{P}\left[\widetilde{\Phi}_{n}^{h,L}\in K_{L}^{c}\right]
  &<\frac{\varepsilon}{2}  \label{ineq:tightdef}.
\end{align}
Since $\Phi_{n}^{h}=\widetilde{\Phi}_{n}^{h,L}$ in $A_{L,T}$, we have 
\begin{eqnarray*}
\mathbb{P}\left[\Phi_{n}^{h}\in K_{L}^{c}\right]
  &\leq &\mathbb{P}\left[\Phi_{n}^{h} \in K_{L}^{c},A_{L,T}\right]+\mathbb{P}\left[A_{L,T}^{c}\right]=\mathbb{P}\left[\widetilde{\Phi}_{n}^{h,L}\in K_{L}^{c},A_{L,T}\right]+\mathbb{P}\left[A_{L,T}^{c}\right] \\
  &\le&  \frac \varepsilon 2+\mathbb{P}\left[A_{L,T}^{c}\right] \le \varepsilon,
\end{eqnarray*}
if $L$ is large enough. This proves property (i) for $g$.

Given $t\in [0,T]$, for every  $\varepsilon>0$ there exists a constant $N_{L}>0$, such that for every $n\geq N_{L}$ and  for every $h=\ell+1,\dots,2\ell$,
\begin{align}\label{ineq:Phiconvdef}
\sup_{n\geq1}\mathbb{P}\left[\Abs{\widetilde{\Phi}_{n}^{h,L}(t)}>\delta\right]
  &<\frac{\varepsilon}{2}.
\end{align}
Again, this implies that
\begin{eqnarray*}
\mathbb{P}\left[ \left|  \Phi_{n}^{h}(t) \right| >\delta \right]
  &\leq &\mathbb{P}\left[  \left|  \Phi_{n}^{h}(t) \right| >\delta,A_{L,T}\right]+\mathbb{P}\left[A_{L,T}^{c}\right]=\mathbb{P}\left[\left| \widetilde{\Phi}_{n}^{h,L} (t) \right| >\delta,A_{L,T}\right]+\mathbb{P}\left[A_{L,T}^{c}\right] \\
  &\le&  \frac \varepsilon 2+\mathbb{P}\left[A_{L,T}^{c}\right] \le \varepsilon,
\end{eqnarray*}
if $L$ is large enough, which proves property (ii) for $g$.

Moreover, if $\alpha=\frac{1}{2\ell+1}$, then for every $0\leq t_{1}\leq \cdots\leq t_{d}\leq T$ there exists $M_{L}\in\N$, such that for all $n\geq M_{L}$,
\begin{align}
\Abs{\mathbb{E}\left[(\phi(\widetilde{\Phi}_{n}^{\ell,L}(t_{1}),\dots, \widetilde{\Phi}_{n}^{\ell,L}(t_{d}))-\phi(\widetilde{Z}_{t_1}^{L},\dots, \widetilde{Z}_{t_d}^{L}))\Indi{B\cap A_{L,T}}\right]}
  &<\frac{\varepsilon}{2},\label{ineq:fddtail}
\end{align}
and if  $t\in [0,T]$ and $\alpha>\frac{1}{2\ell+1}$, there exists $R_{L}\in\N$, such that for all $n\geq R_{L}$,
\begin{align}\label{ineq:Phiconvdef2}
\sup_{n\geq1}\mathbb{P}\left[\Abs{\widetilde{\Phi}_{n}^{\ell,L}(t)}>\delta\right]
  &<\frac{\varepsilon}{2}.
\end{align}
 Similarly, we have
\begin{eqnarray*}
&& \Abs{\mathbb{E}\left[(\phi(\Phi_{n}^{\ell}(t_{1}),\dots, \Phi_{n}^{\ell}(t_{d}))-\phi(Z_{t_{1}},\dots, Z_{t_d}))\Indi{B}\right]}\\
 && \leq \Abs{\mathbb{E}\left[(\phi(\widetilde{\Phi}_{n}^{\ell,L}(t_{1}),\dots, \widetilde{\Phi}_{n}^{\ell,L}(t_{d}))-\phi(\widetilde{Z}_{t_1}^{L},\dots, \widetilde{Z}_{t_d}^{L}))\Indi{B}\Indi{A_{L,T}}\right]}+2\sup_{x\in\mathbb{R}^{d}}\Abs{\phi(x)}\mathbb{P}\left[A_{L,T}^{c}\right] \\
 &&\le \frac {\varepsilon} 2 +2\sup_{x\in\mathbb{R}^{d}}\Abs{\phi(x)}\mathbb{P}\left[A_{L,T}^{c}\right] 
\end{eqnarray*}
and 
\begin{eqnarray*}
\mathbb{P}\left[\Abs{\Phi_{n}^{h}(t)}>\delta\right]
  &\leq & \mathbb{P}\left[\Abs{\Phi_{n}^{h}(t)}>\delta,A_{L,T}\right]+\mathbb{P}\left[A_{L,T}^{c}\right]=\mathbb{P}\left[\Abs{\widetilde{\Phi}_{n}^{h,L}(t)}>\delta,A_{L}\right]+\mathbb{P}\left[A_{L,T}^{c}\right]\\
  &\le &   \frac {\varepsilon} 2 +\mathbb{P}\left[A_{L,T}^{c}\right].
\end{eqnarray*}
Taking $L$ large enough we conclude that properties (iii) and (iv) hold for $g$.

Therefore, we can assume without loss of generality that $f$ has compact support.  Relations (i), (ii) and (iv), for $f$ compactly supported follow from Lemma \ref{lem:tight}, while relation (iii) follows from Lemma  \ref{lem:fdd}. Modulo these two lemmas, which we state below, the proof of Theorem \ref{thm1} is now complete.
%=====================================================================================================
\begin{lemma}\label{lem:tight}
Assume that $\alpha\geq\frac{1}{2h+1}$. Consider the process $\Phi_{n}^{h}$, $h=\ell,\dots, 2\ell$ defined in \eqref{def:Phi}, for $f\in {\cal{C}}^{8\ell+2}(\mathbb{R})$ with compact support. Then,  
\begin{enumerate}
\item The sequence of processes $\{\Phi_{n}^{h}\}_{n\geq1}$, is tight in $\textbf{D}[0,\infty)$, for $h=\ell,\dots, 2\ell$.
\item If $h\geq \ell+1$, then $\Phi_{n}^{h}\stackrel{\mathbb{P}}{\rightarrow}0$, in the topology of $\textbf{D}[0,\infty)$, as $n\rightarrow\infty$.
\item If $\alpha>\frac{1}{2\ell+1}$, then $\Phi_{n}^{\ell}\stackrel{\mathbb{P}}{\rightarrow}0$, in the topology of $\textbf{D}[0,\infty)$, as $n\rightarrow\infty$.
\end{enumerate}
\end{lemma}
\begin{proof}
Fix $h$, $\ell \le h \le 2\ell$. 
As in Section \ref{eq:convergence_variations}, $c_{0,h},\dots,c_{h,h}$ will denote the coefficients of the Hermite expansion of $x^{2h+1}$, namely,  
\begin{align*}
x^{2h+1}
  &=\sum_{u=0}^{h}c_{u,h}H_{2(h-u)+1}(x).
\end{align*}
Then, by \req{hmap}, we can write 
\[ \frac{\DXj^{2h +1}}{\xi_{j,n}^{2h+1}} = \sum_{u=0}^{h} c_{u,h}~H_{2(h-u)+1}\left(\frac{\DXj}{\xi_{j,n}}\right) = \sum_{u=0}^h c_{u,h}~\delta\left(\frac{\partial_{\frac jn}^{\otimes 2(h-u)+1}}{\xi_{j,n}^{2(h-u)+1}}\right).\]
To prove the result, we use the above relation to write the process $\Phi_{n}^{h}$ as a sum of multiple Skorohod integrals plus a remainder term that converges uniformly to zero on compact intervals. Indeed, we can write, for $h=\ell,\dots, 2\ell$, 
\begin{align*}
\Phi_{n}^{h}(t)
  &=\kappa_{\nu,h}\sum_{j=0}^{\floor{nt}-1}\sum_{u=0}^{h}c_{u,h}\xi_{j,n}^{2u}f^{(2h+1)}(\widetilde{X}_{\frac{j}{n}})\delta^{2h+1-2u}(\partial_{\frac{j}{n}}^{\otimes 2h+1-2u}).
\end{align*}
Hence, applying Lemma \ref{lem:Fdelta} with $F=f^{(2h+1)}(\widetilde{X}_{\frac{j}{n}})$, $q=2h+1-2u$ and $u=\partial_{\frac{j}{n}}^{\otimes 2h+1-2u}$, we obtain 
\begin{align}\label{eq:Phitheta}
\Phi_{n}^{h}(t)
  &=\kappa_{\nu,h}\sum_{u=0}^{h}\sum_{r=0}^{2h+1-2u}\binom{2h+1-2u}{r}c_{u,h}\Theta_{u,r}^{n}(t),
\end{align}
where the random variable $\Theta_{u,r}^{n}(t)$, for $h=\ell,\dots, 2\ell$ fixed, is defined by
\begin{align*}
\Theta_{u,r}^{n}(t)
  &=\delta^{2h+1-2u-r}\bigg(\sum_{j=0}^{\floor{nt}-1}\xi_{j,n}^{2u} f^{(2h+1+r)}(\widetilde{X}_{\frac{j}{n}})\partial_{\frac{j}{n}}^{\otimes 2h+1-2u-r}\Ip{\widetilde{\varepsilon}_{\frac{j}{n}},\partial_{\frac{j}{n}}}_{\hten}^{r}\bigg).
\end{align*} 
By \eqref{eq:Phitheta}, we can decompose the process $\Phi_{n}^{h}(t)$, as 
\begin{align}\label{eq:Phidecomp}
\Phi_{n}^{h}(t)
  &=\Psi_{n}^{h}(t)+R_{n}^{h}(t),
\end{align}
where 
\begin{align}\label{eq:psidef}
\Psi_{n}^{h}(t)
  &=\kappa_{\nu,h}\sum_{u=0}^{h}\sum_{r=0}^{2h-2u}\binom{2h+1-2u}{r}c_{u,h}\Theta_{u,r}^{n}(t),
\end{align}
and 
\[
R_{n}^{h}(t)
  =\kappa_{\nu,h}\sum_{u=0}^{h}c_{u,h}\sum_{j=0}^{\floor{nt}-1}\xi_{j,n}^{2u}f^{(4h+2-2u)}(\widetilde{X}_{\frac{j}{n}})\Ip{\widetilde{\varepsilon}_{\frac{j}{n}},\partial_{\frac{j}{n}}}_{\hten}^{2h+1-2u}.
\]
Therefore, to prove the lemma, it suffices to show the following four claims:
\begin{enumerate}
\item[(a)] The process $R_{n}^{h}=\{R_{n}^{h}(t)\}_{t\geq0}$ converges uniformly to zero in $L^{1}(\Omega)$ on compact intervals, namely, 
for each $T>0$,
\begin{align*}%\label{conv:Runifcl1tozero}
\mathbb{E}\left[\sup_{t\in[0,T]}\Abs{R_{n}^{h}(t)}\right]\rightarrow0
\end{align*}
\item[(b)] The process $\Psi_{n}^{h}=\{\Psi_{n}^{h}(t)\}_{t\geq0}$ is tight in $\textbf{D}[0,\infty)$ for all $\ell\leq h\leq 2\ell$.
\item[(c)] The process $\Psi_{n}^{h}=\{\Psi_{n}^{h}(t)\}_{t\geq0}$ converges to zero in $\textbf{D}[0,\infty)$ for $\ell+1\leq h\leq 2\ell$.
\item[(d)] If $\alpha>\frac{1}{2\ell+1}$, then the process $\Psi_{n}^{\ell}=\{\Psi_{n}^{\ell}(t)\}_{t\geq0}$ converges to zero  in probability in $\textbf{D}[0,\infty)$.
\end{enumerate}

\noindent
{\it Proof of claim (a):} Using inequality \eqref{ineq:varpartial}, as well as the fact that $f$ has compact support, we deduce that 
\begin{align*}
\mathbb{E}\left[\sup_{t\in[0,T]}\Abs{R_{n}^{h}(t)}\right]
  &\leq C\sum_{u=0}^{h}\sum_{j=0}^{\floor{nT}-1}\mathbb{E}\left[\Abs{f^{(4h+2-2u)}(\widetilde{X}_{\frac{j}{n}})}\right]\xi_{j,n}^{2u}\Abs{\Ip{\widetilde{\varepsilon}_{\frac{j}{n}},\partial_{\frac{j}{n}}}_{\hten}}^{2h+1-2u}\\
	&\leq C\sum_{u=0}^{h}\sum_{j=0}^{\floor{nT}-1}n^{-\alpha  u}\Abs{\Ip{\widetilde{\varepsilon}_{\frac{j}{n}},\partial_{\frac{j}{n}}}_{\hten}}^{2h+1-2u}.
\end{align*}
Hence, by inequality \eqref{ineq:covpartialepsilon}, there exists a constant $C>0$, such that
\begin{align}
\mathbb{E}\left[\sup_{t\in[0,T]}\Abs{R_{n}^{h,m}(t)}\right]
  &\leq C\sum_{u=0}^{h}n^{-\alpha u-4\beta(h-u)}\nonumber\\
	&=C(n^{-\alpha h}+\sum_{u=0}^{h-1}n^{-\alpha u}n^{-4\beta(h-u)})\leq C(n^{-\alpha h}+hn^{-4\beta})\label{lim:SupR},
\end{align}
which implies that $\sup_{t\in[0,T]}R_{n}^{h}$ converges to zero in $L^{1}(\Omega)$, as required.\\

\noindent
{\it Proof of claims (b), (c) and (d):} 
Since $h\geq \ell$ and $\alpha\geq(2\ell+1)^{-1}$, by the `Billingsley criterion' (see \cite[Theorem 13.5]{Billingsley}), it suffices to show that for every $0\leq s\leq t\leq T$, and $p>2$, there exists a constant $C>0$, such that
\begin{align}\label{ineq:Psibound}
\mathbb{E}\left[\Abs{\Psi_{n}^{h}(t)-\Psi_{n}^{h}(s)}^{p}\right]\leq Cn^{\frac{p}{2}(1-\alpha(2h+1))}\Abs{\frac{\floor{nt}-\floor{ns}}{n}}^{\frac{p}{2}}.
\end{align}
Indeed, relation \eqref{ineq:Psibound} implies that 
$$\mathbb{E}\left[\Abs{\Psi_{n}^{h}(t)-\Psi_{n}^{h}(s)}^{p}\right ] \leq C\Abs{\frac{\floor{nt}-\floor{ns}}{n}}^{\frac{p}{2}},$$
so that $\Psi_{n}^{h}$ is tight. Moreover, if $\ell+1\geq h$ or $\alpha>\frac{1}{2\ell+1}$, then $\E\left[\Abs{\Psi_{n}^{h}(t)-\Psi_{n}^{h}(w)}^{p}\right]\rightarrow0$ as $n\rightarrow\infty$, which implies conditions (c) and (d).

To prove \eqref{ineq:Psibound} we proceed as follows. By \eqref{eq:psidef}, there exists a constant $C>0$, only depending on $h, \nu$ and $T$, such that 
\begin{align}\label{eq:Phiinc}
\Norm{\Psi_{n}^{h}(t)-\Psi_{n}^{h}(s)}_{L^{p}(\Omega)}
  \leq C\max_{\substack{0\leq u\leq h\\0\leq r\leq 2h-2u}}\Norm{\Theta_{u,r}^{n}(t)-\Theta_{u,r}^{n}(s)}_{L^{p}(\Omega)}.
\end{align}
For $0\leq u\leq h$ and $0\leq r\leq 2h-2u$, define the constant $w=2h+1-2u-r\geq1$. By Meyer's inequality \eqref{eq:Meyer}, we have the following bound for the $L^{p}$-norm appearing in the right-hand side of \eqref{eq:Phiinc}.
\begin{multline}\label{ineq:Meyer}
\Norm{\Theta_{u,r}^{n}(t)-\Theta_{u,r}^{n}(s)}_{L^{p}(\Omega)}^2\\
\begin{aligned}
  &=\Norm{\delta^{w}\left(\sum_{j=\floor{ns}}^{\floor{nt}-1}\xi_{j,n}^{2u}f^{(2h+1+r)}(\widetilde{X}_{\frac{j}{n}})\partial_{\frac{j}{n}}^{\otimes w}\Ip{\widetilde{\varepsilon}_{\frac{j}{n}},\partial_{\frac{j}{n}}}_{\hten}^{r}\right)}_{L^{p}(\Omega)}^{2}\\
  &\leq C\sum_{i=0}^{w}\bigg\lVert\sum_{j=\floor{ns}}^{\floor{nt}-1}\xi_{j,n}^{2u}f^{(2h+1+r+i)}(\widetilde{X}_{\frac{j}{n}})\partial_{\frac{j}{n}}^{\otimes w}\otimes \widetilde{\varepsilon}_{\frac{j}{n}}^{\otimes i}\Ip{\widetilde{\varepsilon}_{\frac{j}{n}},\partial_{\frac{j}{n}}}_{\hten}^{r}\bigg\rVert_{L^{p}(\Omega,\hten^{\otimes (w+i)})}^{2}\\
	&= C\sum_{i=0}^{w}\Norm{\Norm{\sum_{j=\floor{ns}}^{\floor{nt}-1}\xi_{j,n}^{2u}f^{(2h+1+r+i)}(\widetilde{X}_{\frac{j}{n}})\partial_{\frac{j}{n}}^{\otimes w}\otimes \widetilde{\varepsilon}_{\frac{j}{n}}^{\otimes i}\Ip{\widetilde{\varepsilon}_{\frac{j}{n}},\partial_{\frac{j}{n}}}_{\hten}^{r}}_{\hten^{\otimes (w+i)}}^2}_{L^{\frac{p}{2}}(\Omega)}.
\end{aligned}
\end{multline} 
From the previous relation, it follows that there exists a constant $C>0$, such that
\begin{eqnarray}   \notag
\Norm{\Theta_{u,r}^{n}(t)-\Theta_{u,r}^{n}(s)}_{L^{p}(\Omega)}^2
  &\leq  & C\sum_{i=0}^{w}\bigg\lVert\sum_{j,k=\floor{ns}}^{\floor{nt}-1}\xi_{j,n}^{2u}\xi_{k,n}^{2u}f^{(2h+1+r+i)}(\widetilde{X}_{\frac{j}{n}})f^{(2h+1+r+i)}(\widetilde{X}_{\frac{k}{n}})\\
 && \!\!\!\!\!\!  \!\!\!\times \Ip{\partial_{\frac{j}{n}},\partial_{\frac{k}{n}}}_{\hten}^{w}\Ip{\widetilde{\varepsilon}_{\frac{j}{n}},\widetilde{\varepsilon}_{\frac{k}{n}}}_{\hten}^{i}\Ip{\widetilde{\varepsilon}_{\frac{j}{n}},\partial_{\frac{j}{n}}}_{\hten}^{r}\Ip{\widetilde{\varepsilon}_{\frac{k}{n}},\partial_{\frac{k}{n}}}_{\hten}^{r}\bigg\rVert_{L^{\frac{p}{2}}(\Omega)}. \label{ineq:normdelta0}
\end{eqnarray}
Since $f$ has compact support, by applying Minkowski inequality and Cauchy-Schwarz inequality in \eqref{ineq:normdelta0}, we deduce that 
\[
\Norm{\Theta_{u,r}^{n}(t)-\Theta_{u,r}^{n}(s)}_{L^{p}(\Omega)}^2\leq C\sum_{i=0}^{w}\sum_{j,k=\floor{ns}}^{\floor{nt}-1}\xi_{j,n}^{2u+r}\xi_{k,n}^{2u+r}\Norm{\widetilde{\varepsilon}_{\frac{j}{n}}}_{\hten}^{i+r}\Norm{\widetilde{\varepsilon}_{\frac{k}{n}}}_{\hten}^{i+r}\Abs{\Ip{\partial_{\frac{j}{n}},\partial_{\frac{k}{n}}}_{\hten}}^{w}.
\]
From here, using the Cauchy Schwarz inequality, it follows that  
\begin{align*}
\Norm{\Theta_{u,r}^{n}(t)-\Theta_{u,r}^{n}(s)}_{L^{p}(\Omega)}^2
  &\leq C\sum_{j,k=\floor{ns}}^{\floor{nt}-1}\xi_{j,n}^{2u+r}\xi_{k,n}^{2u+r}\Abs{\Ip{\partial_{\frac{j}{n}},\partial_{\frac{k}{n}}}_{\hten}}^{w}\\
	&\leq C \sum_{j,k=\floor{ns}}^{\floor{nt}-1}\xi_{j,n}^{2h}\xi_{k,n}^{2h}\Abs{\Ip{\partial_{\frac{j}{n}},\partial_{\frac{k}{n}}}_{\hten}}.
\end{align*}
Consequently,  we get% that $w=2h+1-2u-r$. 
\begin{align}\label{ineq:normdelta}
\Norm{\Theta_{u,r}^{n}(t)-\Theta_{u,r}^{n}(s)}_{L^{p}(\Omega)}^2
  &\leq C\sum_{x=0}^{\floor{nt}-\floor{ns}-1}\sum_{j=\floor{ns}}^{\floor{nt}-1-x}\xi_{j,n}^{2h}\xi_{j+x,n}^{2h}\Abs{\Ip{\partial_{\frac{j}{n}},\partial_{\frac{j+x}{n}}}_{\hten}}.
\end{align}
Then  the estimate (\ref{ineq:Psibound}) will follow from
\begin{align}\label{eq:Gtildebound}
 \xi_{j,n}^{2h}\xi_{j+x,n}^{2h}\Abs{\Ip{\partial_{\frac{j}{n}},\partial_{\frac{j+x}{n}}} _{\hten}} \le Cn^{-\alpha(2h+1)} x^{-1-\delta},
\end{align}
for some $\delta >0$ and for all $x\ge 3$ and $\floor{ns} \le j \le \floor{nt}-1$.
Set 
\[
\widehat{G}(j,j+x)=\xi_{j,n}^{2h}\xi_{j+x,n}^{2h}|\Ip{\partial_{\frac{j}{n}},\partial_{\frac{j+x}{n}}}_{\Hg}|.
\]
 By  considering the  cases $j=0$, $j\geq x+2$ and $1\leq j\leq x+2$, for $x\geq 3$, we obtain the following bounds:\\

\noindent Case $j=0$: 
Using \eqref{ec2} and \eqref{ineq:varpartial}, we get 
\begin{align*}
\widehat{G}(0,x)
  &\leq Cn^{-(2\alpha h+2\beta)}|\phi(x+1)-\phi(x)|\\
	&\leq Cn^{-\alpha(2h+1)}x^{-\nu}.
\end{align*}

\noindent Case $j\geq x+2$:
Using \eqref{ineq:covpartial1}, we deduce  that for every $j\geq x-2$, 
\begin{align*}
\widehat{G}(j,x)
  &\leq Cn^{-2\beta(2h+1)}j^{(2\beta-\alpha)(h+1)}(j+x)^{(2\beta-\alpha)h+\alpha-2}\\
	&\leq Cn^{-2\beta(2h+1)}j^{(2\beta-\alpha)(h+1)}(j+x)^{(2\beta-\alpha)h}x^{\alpha-2}\\
	&\leq Cn^{-2\beta(2h+1)}(j+x)^{(2\beta-\alpha)(2h+1)}x^{\alpha-2}
	=Cn^{-\alpha(2h+1)}x^{\alpha-2}.%\leq C^\prime n^{-\alpha(2h+1)}x^{\alpha-2}.
\end{align*}

\noindent Case $j\leq x+2:$
Using \eqref{ineq:covpartial2}, we deduce that for all $j\leq x-2$, 
\begin{align}\label{eq:GtildeC3s1}
\widehat{G}(j,x)
  &\leq Cn^{-2\beta(2h+1)}j^{(2\beta-\alpha)h+2\beta+\nu-2}(j+x)^{(2\beta-\alpha)h-\nu}.
	%&\leq Cn^{-2\beta(2h+1)}j^{(2\beta-\alpha)h+2\beta+\nu-2}(j+x)^{(2\beta-\alpha)h}x^{-\nu}\\
	%&\leq Cn^{-2\beta(2h+1)}j^{(2\beta-\alpha)h+2\beta}(j+x)^{(2\beta-\alpha)h}x^{-\nu},
\end{align}
If $\nu\geq 2-\alpha$, then 
$$(j+x)^{-\nu}=(j+x)^{\alpha-2}(j+x)^{2-\alpha-\nu}\leq x^{\alpha-2}j^{2-\alpha-\nu},$$
and thus, by \eqref{eq:GtildeC3s1}, 
\begin{align}\label{eq:GtildeC3s2}
\hat{G}(j,x)
  &\leq Cn^{-2\beta(2h+1)}j^{(2\beta-\alpha)(h+1)}(j+x)^{(2\beta-\alpha)h}x^{\alpha-2}\leq Cn^{-\alpha(2h+1)}x^{\alpha-2}.
\end{align}
On the other hand, if $\nu\leq 2-\alpha$, then by \eqref{eq:GtildeC3s1}, 
\begin{align}\label{eq:GtildeC3s3}
\hat{G}(j,x)
  &\leq Cn^{-2\beta(2h+1)}j^{(2\beta-\alpha)h+2\beta-\alpha}(j+x)^{(2\beta-\alpha)h-\nu}\nonumber\\
	&\leq Cn^{-2\beta(2h+1)}j^{(2\beta-\alpha)(h+1)}(j+x)^{(2\beta-\alpha)h}x^{-\nu}\nonumber\\
	&\leq Cn^{-\alpha(2h+1)}x^{-\nu}.
\end{align}
The proof of the lemma is now complete.
\end{proof}
%======================================
\begin{lemma}\label{lem:fdd}
Assume that $\alpha=\frac{1}{2\ell+1}$ and let $0\leq t_{1}\leq\cdots\leq t_{d}\leq T$ be fixed. Define $\Phi_{n}^{\ell}$ and $Z$ by \eqref{def:Phi} and \eqref{def:Z} respectively, for some function $f$ with compact support. Then,  
\begin{align}\label{conv:convintegral}
(\Phi_{n}^{\ell}(t_{1}),\dots, \Phi_{n}^{\ell}(t_{d}))
	&\stackrel{stably}{\rightarrow}(Z_{t_{1}},\dots, Z_{t_d}).
\end{align}
\end{lemma}
\begin{proof}
We follow the small blocks-big blocks methodology (see \cite{BNN} and \cite{CoNuWo}). Let $2\leq p<n$. For $k\geq0$, define the set 
\begin{align*}
I_{k}
  &=\{j\in\{0,\dots, \floor{nt}-1\}\ |\ \frac{k}{p}\leq\frac{j}{n}< \frac{k+1}{p}\}.
\end{align*}
The basic idea of the proof of \eqref{conv:convintegral}, consists on approximating $(\Phi_{n}^{\ell}(t_{1}),\dots, \Phi_{n}^{\ell}(t_{d}))$ by the random vector $(\widetilde{\Phi}_{n,p}(t_{1}),\dots, \widetilde{\Phi}_{n,p}(t_{d}))$, where 
\begin{align*}
\widetilde{\Phi}_{n,p}(t)
  &=\kappa_{\nu,\ell}\sum_{k=0}^{\floor{pt}}\sum_{j\in I_{k}}f^{(2\ell+1)}(X_{\frac{k}{p}})(\DXj)^{2\ell+1}.
\end{align*}
By Proposition \ref{Variations}, for every $\Fc$-measurable and bounded random variable $\eta$,  the vector $(\widetilde{\Phi}_{n,p}(t_{1}),\dots,\widetilde{\Phi}_{n,p}(t_{d}),\eta)$ converges in law, as $n$ tends to infinity, to the vector $(\Xi_{p}^{1},\dots, \Xi_{p}^{d},\eta)$, where
\begin{align*}
\Xi_{p}^{i}
  &=\kappa_{\nu,\ell}\sigma_{\ell}\sum_{k=0}^{\floor{pt_{i}}}f^{(2\ell+1)}(X_{\frac{k}{p}})(Y_{\frac{k+1}{p}}-Y_{\frac{k}{p}}),\ \ \ \ \ \ \ \ \ \text{for}\ i=1,\dots, d.
\end{align*} 
In turn, when $p\rightarrow\infty$, the random vector $(\Xi_{p}^{1},\dots,\Xi_{p}^{d},\eta)$ converges in probability to a random vector with the same law as $(Z_{t_{1}},\dots, Z_{t_{d}},\eta)$, which  implies \eqref{conv:convintegral}, provided that 
\begin{align}\label{lim:PhiPhitilde}
\lim_{p\rightarrow\infty}\limsup_{n\rightarrow\infty}\sum_{i=0}^{d}\Norm{\Phi_{n}^{\ell}(t_{i})-\widetilde{\Phi}_{n,p}(t_{i})}_{L^{2}(\Omega)}
  &=0.
\end{align}
Indeed, if \eqref{lim:PhiPhitilde} holds, then for all $g:\R^{d+1}\rightarrow\R$ differentiable with compact support, and every $p\geq 1$,
\begin{multline*}
\limsup_{n\rightarrow\infty}\Abs{\E\left[g(\Phi_{n}^{\ell}(t_{1}),\dots, \Phi_{n}^{\ell}(t_{d}),\eta)-g(Z_{t_{1}},\dots, Z_{t_{d}},\eta)\right]}\\
\begin{aligned}
&\leq\limsup_{n\rightarrow\infty}\Abs{\E\left[g(\Phi_{n}^{\ell}(t_{1}),\dots, \Phi_{n}^{\ell}(t_{d}),\eta)-g(\widetilde{\Phi}_{n,p}(t_{1}),\dots, \widetilde{\Phi}_{n,p}(t_{d}),\eta)\right]}\\
&+\limsup_{n\rightarrow\infty}\Abs{\E\left[g(\widetilde{\Phi}_{n,p}(t_{1}),\dots, \widetilde{\Phi}_{n,p}(t_{d}),\eta)-g(Z_{t_{1}},\dots, Z_{t_{d}},\eta)\right]}\\
&=\limsup_{n\rightarrow\infty}\Abs{\E\left[g(\Phi_{n}^{\ell}(t_{1}),\dots, \Phi_{n}^{\ell}(t_{d}),\eta)-g(\widetilde{\Phi}_{n,p}(t_{1}),\dots, \widetilde{\Phi}_{n,p}(t_{d}),\eta)\right]}\\
&+\Abs{\E\left[g(\Xi_{p}^{1},\dots, \Xi_{p}^{d},\eta)-g(Z_{t_{1}},\dots, Z_{t_{d}},\eta)\right]}.
\end{aligned}
\end{multline*}
Then, taking $p\rightarrow\infty$, we get
$$\lim_{n\rightarrow \infty} \left|\E\left[g(\Phi_{n}^{\ell}(t_{1}),\dots, \Phi_{n}^{\ell}(t_{d}),\eta)\right] - \E\left[g(z_{t_{1}},\dots, Z_{t_{d}},\eta)\right] \right| =0,$$
as required.

In order to prove \eqref{lim:PhiPhitilde} we proceed as follows. Following the proof of \eqref{eq:Phitheta}, we can show that 
\begin{align}
\Phi_{n}^{\ell}(t_{i})
  &=\kappa_{\nu,\ell}\sum_{u=0}^{\ell}\sum_{r=0}^{2\ell+1-2u}\Comb{2\ell+1-2u\\r}c_{u,\ell}\Theta_{u,r}^{n}(t_{i}),\label{eq:Phiexpansion}\\
\widetilde{\Phi}_{n,p}(t_{i})
  &=\kappa_{\nu,\ell}\sum_{u=0}^{\ell}\sum_{r=0}^{2\ell+1-2u}\Comb{2\ell+1-2u\\r}c_{u,\ell}\widetilde{\Theta}_{u,r}^{n,p}(t_{i}),\label{eq:Phitildeexpansion}
\end{align}
where $\Theta_{u,r}^{n}(t)$ and $\widetilde{\Theta}_{u,r}^{n,p}(t)$ are defined, for $0\leq u\leq \ell$ and $0\leq r\leq 2\ell+1-2u$, by
\begin{align*}
\Theta_{u,r}^{n}(t)
  &=\delta^{2\ell+1-2u-r}\bigg(\sum_{k=0}^{\floor{pt_{i}}}\sum_{j\in I_{k}}\xi_{j,n}^{2u} f^{(2\ell+1+r)}(\widetilde{X}_{\frac{j}{n}})\partial_{\frac{j}{n}}^{\otimes 2\ell+1-2u-r}\Ip{\widetilde{\varepsilon}_{\frac{j}{n}},\partial_{\frac{j}{n}}}^{r}\bigg),\\
\widetilde{\Theta}_{u,r}^{n,p}(t)
  &=\delta^{2\ell+1-2u-r}\bigg(\sum_{k=0}^{\floor{pt_{i}}}\sum_{j\in I_{k}}\xi_{j,n}^{2u} f^{(2\ell+1+r)}(X_{\frac{k}{p}})\partial_{\frac{j}{n}}^{\otimes 2\ell+1-2u-r}\Ip{\varepsilon_{\frac{k}{p}},\partial_{\frac{j}{n}}}^{r}\bigg).
\end{align*}
In view of \eqref{eq:Phiexpansion} and \eqref{eq:Phitildeexpansion}, relation \eqref{lim:PhiPhitilde} holds true, provided that we show that for every $t\geq0$
\begin{align}\label{conv:PhiL2}
\lim_{p\rightarrow\infty}\limsup_{n\rightarrow\infty}\Norm{\Theta_{u,r}^{n}(t)-\widetilde{\Theta}_{u,r}^{n,p}(t)}_{L^{2}(\Omega)}
  &=0.
\end{align}
We divide the proof of \eqref{conv:PhiL2} in several steps.\\

\noindent
\textit{Step 1.} 
First we prove \eqref{conv:PhiL2} in the case $r=2\ell+1-2u$. To this end, it suffices to show that for every $p$ fixed,
\begin{align}
\lim_{n\rightarrow\infty}\Norm{\sum_{k=0}^{\floor{pt}}\sum_{j\in I_{k}}\xi_{j,n}^{2u} f^{(4\ell+2-2u)}(X_{\frac{k}{p}})\Ip{\varepsilon_{\frac{k}{p}},\partial_{\frac{j}{n}}}_{\Hg}^{2\ell+1-2u}}_{L^{2}(\Omega)}
  &=0,\label{lim:Thetareminder1}
\end{align}
and 
\begin{align}
\lim_{n\rightarrow\infty}\Norm{\sum_{k=0}^{\floor{pt}}\sum_{j\in I_{k}}\xi_{j,n}^{2u} f^{(4\ell+2-2u)}(\widetilde{X}_{\frac{j}{n}})\Ip{\widetilde{\varepsilon}_{\frac{j}{n}},\partial_{\frac{j}{n}}}_{\Hg}^{2\ell+1-2u}}_{L^{2}(\Omega)}
  &=0.\label{lim:Thetareminder2}
\end{align}
Relation \eqref{lim:Thetareminder2} was already proved in Lemma \ref{lem:tight} (see inequality \eqref{lim:SupR}). In order to prove \eqref{lim:Thetareminder1} we proceed as follows. Since $f$ has compact support, there exists a constant $C>0$, such that for every $u=0,\dots, \ell$, we have
\begin{multline*}
\Norm{\sum_{k=0}^{\floor{pt}}\sum_{j\in I_{k}}\xi_{j,n}^{2u} f^{(4\ell+2-2u)}(X_{\frac{k}{p}})\Ip{\varepsilon_{\frac{k}{p}},\partial_{\frac{j}{n}}}_{\hten}^{2\ell+1-2u}}_{L^{2}(\Omega)}\\
\begin{aligned}
  &\leq C\sum_{k=0}^{\floor{pt}}\sum_{j\in I_{k}}\xi_{j,n}^{2u} \Abs{\Ip{\varepsilon_{\frac{k}{p}},\partial_{\frac{j}{n}}}_{\hten}}^{2\ell+1-2u}\\
	&\leq C\left(\frac{k}{p}\right)^{2\beta(\ell-u)}\phi(1)^{\ell-u}\sum_{k=0}^{\floor{pt}}\sum_{j\in I_{k}}\xi_{j,n}^{2\ell} \Abs{\Ip{\varepsilon_{\frac{k}{p}},\partial_{\frac{j}{n}}}_{\hten}},
\end{aligned}
\end{multline*}
where the last inequality follows from Cauchy-Schwarz inequality  and \eqref{eq:R}. Therefore, by relation \eqref{ineq:varpartial} 
 there exist a constant $C_{k,p}>0$, such that
\begin{multline*}
\Norm{\sum_{k=0}^{\floor{pt}}\sum_{j\in I_{k}}\xi_{j,n}^{2u} f^{(4\ell+2-2u)}(X_{\frac{k}{p}})\Ip{\varepsilon_{\frac{k}{p}},\partial_{\frac{j}{n}}}_{\hten}^{2\ell+1-2u}}_{L^{2}(\Omega)} \\
\begin{aligned} &\le C_{k,p}  n^{-\alpha \ell}\sum_{k=1}^{\floor{pt}} \sum_{j\in I_{k}}\left( \frac kp \right) ^{2\beta} \left| \phi\left(\frac{(j+1) p}{nk}\right) -\phi\left(\frac {jp}{nk}\right) \right|.
\end{aligned}
\end{multline*}
Using the decomposition (\ref{phi1}) we get
\begin{eqnarray*}
 \left| \phi\left(\frac{(j+1) p}{nk}\right) -\phi\left(\frac {jp}{nk}\right) \right| & \le&  \lambda\left[ \left(\frac {(j+1)p}{nk} -1\right)^\alpha -\left(\frac {jp}{nk} -1\right)^\alpha \right]\\
 &&
 + \left| \psi \left( \frac {(j+1)p}{nk} \right) -\psi\left( \frac {jp}{nk} \right) \right|\\
 &\le &  \lambda\left[ \left(\frac {(j+1)p}{nk} -1\right)^\alpha -\left(\frac {jp}{nk} -1\right)^\alpha \right ]+ \sup_{x\ge 1} |\psi'(x)| \frac p{nk}.
 \end{eqnarray*}
The sum in $j\in I_k$ of this expression is bounded by a constant not depending on $n$ because the first term produces a telescopic sum and the second term is bounded by a constant times $1/n$.  This completes the proof of the convergence  \eqref{lim:Thetareminder1}.

\noindent
\textit{Step 2.} 
Next we show \eqref{conv:PhiL2} for $0\leq r\leq 2\ell-2u$. To this end, define the variables
$$F_{k,j,r}^{n,p}=f^{(2\ell+1+r)}(\widetilde{X}_{\frac{j}{n}})\Ip{\widetilde{\varepsilon}_{\frac{j}{n}},\partial_{\frac{j}{n}}}^{r}-f^{(2\ell+1+r)}(X_{\frac{k}{p}})\Ip{\varepsilon_{\frac{k}{p}},\partial_{\frac{j}{n}}}^{r}.$$ 
We aim to show that for every $u=0,\dots, \ell$, and $0\leq r\leq 2\ell-2u$,
\begin{align}\label{conv:Meyerl}
\lim_{p\rightarrow\infty}\limsup_{n\rightarrow\infty}\Norm{\delta^{2\ell+1-2u-r}\left(\sum_{k=0}^{\floor{pt}}\sum_{j\in I_{k}}\xi_{j,n}^{2u}F_{k,j,r}^{n,p} \partial_{\frac{j}{n}}^{\otimes 2\ell+1-2u-r}\right)}_{L^{2}(\Omega)}=0.
\end{align}
Define $w=2\ell+1-2u-r$. By Meyer's inequality \eqref{eq:Meyer}, we have 
\begin{multline}\label{ineq:Meyerl}
\Norm{\delta^{w}\left(\sum_{k=0}^{\floor{pt}}\sum_{j\in I_{k}}\xi_{j,n}^{2u}F_{k,j,r}^{n,p} \partial_{\frac{j}{n}}^{\otimes w}\right)}_{L^{2}(\Omega)}^{2}\\
\begin{aligned}
  &\leq C\sum_{i=0}^{w}\Norm{\sum_{k=0}^{\floor{pt}}\sum_{j\in I_{k}}\xi_{j,n}^{2u}D^{i}F_{k,j,r}^{n,p}\otimes\partial_{\frac{j}{n}}^{\otimes w}}_{L^{2}(\Omega;\hten^{\otimes (w+i)})}^2\\
	&= C\sum_{i=0}^{w}\sum_{k_{1},k_{2}=0}^{\floor{pt}}\sum_{\substack{j_{1}\in I_{k_{1}}\\j_{2}\in I_{k_{2}}}}\xi_{j_{1},n}^{2u}\xi_{j_{2},n}^{2u}
	\mathbb{E}\left[\Ip{D^{i}F_{k_{1},j_{1},r}^{n,p},D^{i}F_{k_{2},j_{2},r}^{n,p}}_{\hten^{\otimes i}}\right]
	\Ip{\partial_{\frac{j_{1}}{n}},
	\partial_{\frac{j_{2}}{n}}}_{\hten}^{w}.
\end{aligned}
\end{multline}
By the Cauchy-Schwarz inequality, we have $\Abs{\Ip{\partial_{\frac{j_{1}}{n}},\partial_{\frac{j_{2}}{n}}}_{\hten}}\leq\xi_{j_{1},n}\xi_{j_{2},n}$, and hence, 
\begin{align*}
\Abs{\Ip{\partial_{\frac{j_{1}}{n}},\partial_{\frac{j_{2}}{n}}}_{\hten}}^{w}
  &\leq (\xi_{j_{1},n}\xi_{j_{2},n})^{2\ell-2u-r}\Abs{\Ip{\partial_{\frac{j_{1}}{n}},\partial_{\frac{j_{2}}{n}}}_{\hten}},
\end{align*}
which, by \eqref{ineq:Meyerl}, implies that 
\begin{multline}\label{ineq:Meyerl2}
\Norm{\delta^{w}\left(\sum_{k=0}^{\floor{pt}}\sum_{j\in I_{k}}\xi_{j,n}^{2u}F_{k,j,r}^{n,p} \partial_{\frac{j}{n}}^{\otimes q}\right)}_{L^{2}(\Omega)}^{2}\\
\begin{aligned}
  &\leq \sum_{i=0}^{w}\sum_{k_{1},k_{2}=0}^{\floor{pt}}\sum_{\substack{j_{1}\in I_{k_{1}}\\j_{2}\in I_{k_{2}}}}\xi_{j_{1},n}^{2\ell-r}\xi_{j_{2},n}^{2\ell-r}
	\Norm{D^{i}F_{k_{1},j_{1},r}^{n,p}}_{L^{2}(\Omega,\hten^{\otimes i})}\Norm{D^{i}F_{k_{2},j_{2},r}^{n,p}}_{L^{2}(\Omega,\hten^{\otimes i})}
	\Abs{\Ip{\partial_{\frac{j_{1}}{n}},\partial_{\frac{j_{2}}{n}}}_{\hten}}\\
	&\leq \sum_{i=0}^{w}\max_{(k,j)\in J_{n,p}}\Norm{D^{i}F_{k,j,r}^{n,p}}_{L^{2}(\Omega,\hten^{\otimes i})}^2\sum_{k_{1},k_{2}=0}^{\floor{pt}}\sum_{\substack{j_{1}\in I_{k_{1}}\\j_{2}\in I_{k_{2}}}}\xi_{j_{1},n}^{2\ell-r}\xi_{j_{2},n}^{2\ell-r}
	\Abs{\Ip{\partial_{\frac{j_{1}}{n}},\partial_{\frac{j_{2}}{n}}}_{\hten}},
\end{aligned}
\end{multline}
where $J_{n,p}$ denotes the set of indices 
$$J_{n,p}=\{(k,j)\in\mathbb{N}\ |\ 0\leq k\leq  \floor{pt} +1~~\text{and}~~\frac{k}{p}\leq\frac{j}{n}\leq \frac{k+1}{p}\}.$$
We can easily check that
\begin{align*}
F_{k,j,r}^{n,p}
  &= f^{(2\ell+1+r)}(\widetilde{X}_{\frac{j}{n}})\Ip{\widetilde{\varepsilon}_{\frac{j}{n}}^{\otimes r}-\varepsilon_{\frac{k}{p}}^{\otimes r},\partial_{\frac{j}{n}}^{\otimes r}}_{\hten^{\otimes r}}\\
	&+\left(f^{(2\ell+1+r)}(\widetilde{X}_{\frac{j}{n}})-f^{(2\ell+1+r)}(X_{\frac{k}{p}})\right)\Ip{\varepsilon_{\frac{k}{p}},\partial_{\frac{j}{n}}}_{\Hg}^{r},
\end{align*}
and hence, we have
\begin{align*}
D^{i}F_{k,j,r}^{n,p}
	&= f^{(2\ell+1+r+i)}(\widetilde{X}_{\frac{j}{n}})\widetilde{\varepsilon}_{\frac{j}{n}}^{\otimes i}\Ip{\widetilde{\varepsilon}_{\frac{j}{n}}^{\otimes r}-\varepsilon_{\frac{k}{p}}^{\otimes r},\partial_{\frac{j}{n}}^{\otimes r}}_{\hten^{\otimes r}}\\
	&+f^{(2\ell+1+r+i)}(\widetilde{X}_{\frac{j}{n}})\left(\widetilde{\varepsilon}_{\frac{j}{n}}^{\otimes i}-\varepsilon_{\frac{k}{p}}^{\otimes i}\right)\Ip{\varepsilon_{\frac{k}{p}},\partial_{\frac{j}{n}}}_{\Hg}^{r}\\
	&+\left(f^{(2\ell+1+r+i)}(\widetilde{X}_{\frac{j}{n}})-f^{(2\ell+1+r+i)}(X_{\frac{k}{p}})\right)\varepsilon_{\frac{k}{p}}^{\otimes i}\Ip{\varepsilon_{\frac{k}{p}},\partial_{\frac{j}{n}}}_{\Hg}^{r}.
\end{align*}
From the previous equality, and the compact support condition of $f$, we deduce that there exists a constant $C>0$, such that
\begin{multline*}
\Norm{D^{i}F_{k,j,r}^{n,p}}_{L^{2}(\Omega, \hten^{\otimes i})}\\
\begin{aligned}
	&\leq C\Norm{\widetilde{\varepsilon}_{\frac{j}{n}}^{\otimes i}}_{\hten^{\otimes i}}\Abs{\Ip{\widetilde{\varepsilon}_{\frac{j}{n}}^{\otimes r}-\varepsilon_{\frac{k}{p}}^{\otimes r},\partial_{\frac{j}{n}}^{\otimes r}}_{\hten^{\otimes r}}}+C\Norm{\widetilde{\varepsilon}_{\frac{j}{n}}^{\otimes i}-\varepsilon_{\frac{k}{p}}^{\otimes i}}_{\hten^{\otimes i}}\Abs{\Ip{\varepsilon_{\frac{k}{p}},\partial_{\frac{j}{n}}}_{\Hg}^{r}}\nonumber\\
	&+\Norm{f^{(2\ell+1+r+i)}(\widetilde{X}_{\frac{j}{n}})-f^{(2\ell+1+r+i)}(X_{\frac{k}{p}})}_{L^{2}(\Omega)}\Norm{\varepsilon_{\frac{k}{p}}^{\otimes i}}_{\hten^{\otimes i}}\Abs{\Ip{\varepsilon_{\frac{k}{p}},\partial_{\frac{j}{n}}}_{\Hg}^{r}},
\end{aligned}
\end{multline*}
and hence, 
\begin{align}\label{ineq:DFkjr}
\Norm{D^{i}F_{k,j,r}^{n,p}}_{L^{2}(\Omega, \hten^{\otimes i})}
	&\leq C\Norm{\widetilde{\varepsilon}_{\frac{j}{n}}}_{\hten}^{i}\Norm{\widetilde{\varepsilon}_{\frac{j}{n}}^{\otimes r}-\varepsilon_{\frac{k}{p}}^{\otimes r}}_{\hten^{\otimes r}}\Norm{\partial_{\frac{j}{n}}}_{\hten}^{r}\\
	&+C\Norm{\widetilde{\varepsilon}_{\frac{j}{n}}^{\otimes i}-\varepsilon_{\frac{k}{p}}^{\otimes i}}_{\hten^{\otimes i}}\Norm{\varepsilon_{\frac{k}{p}}}_{\hten}^{r}\Norm{\partial_{\frac{j}{n}}}_{\hten}^{r}\nonumber\\
	&+\Norm{f^{(2\ell+1+r+i)}(\widetilde{X}_{\frac{j}{n}})-f^{(2\ell+1+r+i)}(X_{\frac{k}{p}})}_{L^{2}(\Omega)}\Norm{\varepsilon_{\frac{k}{p}}}_{\hten}^{r+i}\Norm{\partial_{\frac{j}{n}}}_{\hten}^{r}.\nonumber
\end{align}
Using the Cauchy-Schwarz inequality, as well as \eqref{eq:R}, we have that for every $\gamma\in\N$, $\gamma\geq1$, there exists a constant $C>0$ such that
\begin{align*}
\Norm{\widetilde{\varepsilon}_{\frac{j}{n}}^{\otimes \gamma}-\varepsilon_{\frac{k}{p}}^{\otimes \gamma}}_{\hten^{\otimes \gamma}}
  \leq \Norm{\widetilde{\varepsilon}_{\frac{j}{n}}-\varepsilon_{\frac{k}{p}}}_{\hten}
	\sum_{i=0}^{\gamma-1}\Norm{\widetilde{\varepsilon}_{\frac{j}{n}}}_{\hten}^{i}\Norm{\varepsilon_{\frac{k}{p}}}_{\hten}^{\gamma-1-i}
	\leq C\Norm{\widetilde{\varepsilon}_{\frac{j}{n}}-\varepsilon_{\frac{k}{p}}}_{\hten}.
\end{align*}
As a consequence, by \eqref{ineq:DFkjr}, there exists a constant $C>0$ such that
\begin{multline*}
\Norm{D^{i}F_{k,j,r}^{n,p}}_{L^{2}(\Omega, \hten^{\otimes i})}\\
\begin{aligned}
	&\leq C\xi_{j,n}^{r}\bigg(\Norm{\widetilde{\varepsilon}_{\frac{j}{n}}-\varepsilon_{\frac{k}{p}}}_{\hten}+\Norm{f^{(2\ell+1+r+i)}(\widetilde{X}_{\frac{j}{n}})-f^{(2\ell+1+r+i)}(X_{\frac{k}{p}})}_{L^{2}(\Omega)}\bigg)\\
	&\leq C\xi_{j,n}^{r}\bigg(\mathbb{E}\left[\sup_{\Abs{t-s}\leq \frac{1}{p}}\Abs{\widetilde{X}_{t}-X_{s}}^2\right]^{\frac{1}{2}}
	+\mathbb{E}\left[\sup_{\Abs{t-s}\leq\frac{1}{p}}\Abs{f^{(2\ell+1+r+i)}(\widetilde{X}_{t})-f^{(2\ell+1+r+i)}(X_{s})}^2\right]^{\frac{1}{2}}\bigg).
\end{aligned}
\end{multline*}
From the previous inequality, we deduce that the function
\begin{align*}
Q_{p}=\sup_{n\geq 1}\xi_{j,n}^{-2r}\sum_{i=0}^{q}\max_{(k,j)\in J_{n,p}}\Norm{D^{i}F_{k,j,r}^{n,p}}_{L^{2}(\Omega, \hten^{\otimes i})}^2,
\end{align*}
satisfies $\lim_{p\rightarrow\infty}Q_{p}=0$. Hence, by \eqref{ineq:varpartial} and \eqref{ineq:Meyerl2},
\begin{align}\label{ineq:Meyerl3}
\Norm{\delta^{q}\left(\sum_{k=0}^{\floor{pt}}\sum_{j\in I_{k}}\xi_{j,n}^{2u}F_{k,j,r}^{n,p} \partial_{\frac{j}{n}}^{\otimes q}\right)}_{L^{2}(\Omega)}^{2}
  &\leq C Q_{p}\sum_{k_{1},k_{2}=0}^{\floor{pt}}\sum_{\substack{j_{1}\in I_{k_{1}}\\j_{2}\in I_{k_{2}}}}\xi_{j_{1},n}^{2\ell}\xi_{j_{2},n}^{2\ell}\nonumber
	\Abs{\Ip{\partial_{\frac{j_{1}}{n}},\partial_{\frac{j_{2}}{n}}}_{\hten}}\\
	&= C Q_{p}\sum_{i_{1},i_{2}=0}^{\floor{nt}}\xi_{i_{1},n}^{2\ell}\xi_{i_{2},n}^{2\ell}
	\Abs{\Ip{\partial_{\frac{i_{1}}{n}},\partial_{\frac{i_{2}}{n}}}_{\hten}}\nonumber\\
	&\leq C Q_{p}\sum_{x=0}^{\floor{nt}-1}\sum_{j=0}^{\floor{nt}-1-x}\xi_{j,n}^{2\ell}\xi_{j+x,n}^{2\ell}
	\Abs{\Ip{\partial_{\frac{j}{n}},\partial_{\frac{j+x}{n}}}_{\hten}}.
\end{align}

\noindent Using the previous inequality, as well as \eqref{eq:Gtildebound}, we deduce that
\begin{align}\label{eq:SkorFkjr}
\Norm{\delta^{q}\left(\sum_{k=0}^{\floor{pt}}\sum_{j\in I_{k}}\xi_{j,n}^{2u}F_{k,j,r}^{n,p} \partial_{\frac{j}{n}}^{\otimes q}\right)}_{L^{2}(\Omega)}^{2}
  &\leq C t Q_{p}\sum_{x=0}^{\infty}n^{1-\alpha(2\ell+1)}(1+x)^{-1-\delta},
\end{align}
for some $\delta>0$. %By \eqref{eq:SkorFkjr} as well as the fact that
Since,  $\alpha=\frac{1}{2\ell+1}$, relation \eqref{eq:SkorFkjr} implies that

%Define $j=i_{1}$ and $x=i_{2}-i_{1}$. Then, using %inequality $\xi_{j,n}^{-1}\xi_{j+x,n}^{-1}
	%%\Abs{\Ip{\partial_{\frac{j}{n}},\partial_{\frac{j+x}{n}}}_{\hten}}\leq1$, and 
	%the condition $\alpha=(2\ell+1)^{-1}$, we get
%\begin{align*}
%\Norm{\delta^{q}\left(\sum_{k=0}^{\floor{pt}}\sum_{j\in I_{k}}\xi_{j,n}^{2u}F_{k,j,r}^{n,p} \partial_{\frac{j}{n}}^{\otimes q}\right)}_{L^{2}(\Omega)}^{2}
  %&\leq \frac{2C Q_{p}}{n}\sum_{j=0}^{\floor{nt}}
	%\sum_{x=0}^{\floor{nt}-j}
	%\xi_{j,n}^{-1}\xi_{j+x,n}^{-1}
	%\Abs{\Ip{\partial_{\frac{j}{n}},\partial_{\frac{j+x}{n}}}_{\hten}}\\
	%&\leq 2Ct Q_{p}
	%\sum_{x=0}^{\floor{nt}}
	%\sup_{0\leq j\leq \floor{nt}}\xi_{j,n}^{-1}\xi_{j+x,n}^{-1}
	%\Abs{\Ip{\partial_{\frac{j}{n}},\partial_{\frac{j+x}{n}}}_{\hten}}.
%\end{align*}
%By \eqref{ineq:supcorrelation}, there exist $\delta,C>0$, such that
%\begin{align*}
%\Norm{\delta^{q}\left(\sum_{k=0}^{\floor{pt}}\sum_{j\in I_{k}}\xi_{j,n}^{2u}F_{k,j,r}^{n,p} \partial_{\frac{j}{n}}^{\otimes q}\right)}_{L^{2}(\Omega)}^{2}
  %&\leq Ct Q_{p}
	%\sum_{x=0}^{\infty}(1+x)^{-1-\delta}.
%\end{align*}
%Since the series in the right hand side is convergent, we deduce that there is a constant $C>0$, such that 
\begin{align}\label{ineq:Meyerlfinal}
\Norm{\delta^{q}\left(\sum_{k=0}^{\floor{pt}}\sum_{j\in I_{k}}\xi_{j,n}^{2u}F_{k,j,r}^{n,p} \partial_{\frac{j}{n}}^{\otimes q}\right)}_{L^{2}(\Omega)}^{2}
  &\leq CtQ_{p}.
\end{align}
Relation \eqref{conv:Meyerl} then follows from \eqref{ineq:Meyerlfinal} since $\lim_{p\rightarrow\infty}Q_{p}=0$. The proof is now complete.
\end{proof}

%\subsection{Proof of Corollary \ref{Cor1}}
%========================================================================================================================
\section{Appendix} \label{Sec:App}
The following lemmas are estimations on the covariances of increments of $X$. The proof of these results relies on some technical lemmas proved by Nualart and Harnett in \cite{Hermite}.  In what follows $C$ is a generic constant depending only on the covariance of the process $X$.

\begin{lemma}  \label{lem3.2}Under (H.1),  for $0<s\le t$ we have
\[
{\mathbb E}\left[ (X_{t+s} - X_t)^2\right] = 2\lambda t^{2\beta -\alpha}s^\alpha + g_1(t,s),
\]
where $|g_1(t,s)| \le C st^{2\beta-1}$.
\end{lemma}
\begin{proof}
See \cite[Lemma~3.1]{Hermite} and notice that the proof only uses that  $|\psi'|$ is bounded in $(1,2]$.
\end{proof}
\begin{remark}  Notice that $g_1(t,s)$ satisfies
$|g_1(t,s)| \le  C s^{\alpha} t^{2\beta-\alpha }$,
because  $\alpha<1$ and $\alpha\leq2\beta$.  Therefore, for any $0<s\le t$, we obtain
\[
{\mathbb E}\left[ (X_{t+s} - X_t)^2\right] \le  C s^{\alpha} t^{2\beta-\alpha }.
\]
With the notation of Section 2.3, this implies
\begin{equation} \label{ecua1}
\xi^2_{j,n} \le Cn^{-2\beta} j^{2\beta-\alpha}.
\end{equation}
On the other hand, we deduce that for every $T>0$, there exists $C>0$, which depends on $T$ and the covariance of $X$, such that 
\begin{align}\label{ineq:varpartial}
\sup_{0\leq t\leq \floor{nT}}\mathbb{E}\left[\Delta X_{\frac{t}{n}}^{2}\right]
\leq C n^{-\alpha}.
\end{align}
\end{remark}

\begin{lemma}\label{lem:cov}
Let $j, k, n$ be integers with $n\geq6$ and $1\leq j\leq k$. Under (H.1)-(H.2), we have the following estimates:
\begin{enumerate}[(a)]
\item If $j+3\leq k\leq 2j+2$, then
\begin{align}\label{ineq:covpartial1}
\Abs{\mathbb{E}\left[\Delta X_{\frac{j}{n}}\Delta X_{\frac{k}{n}}
\right]
}
  \leq Cn^{-2\beta}j^{2\beta-\alpha}k^{\alpha-2}.
\end{align}
\item If $k\geq 2j+2$, then
\begin{align}\label{ineq:covpartial2}
\Abs{\mathbb{E}\left[\Delta X_{\frac{j}{n}}\Delta X_{\frac{k}{n}}
\right]
}
  &\leq Cn^{-2\beta}j^{2\beta+\nu-2}k^{-\nu}.
\end{align}
\end{enumerate}
\end{lemma}
\begin{proof}
We have
\begin{align*}
{\mathbb E}\left[ \Delta X_{\frac kn} \Delta X_{\frac{j}{n}}\right] &= n^{-2\beta}(j+1)^{2\beta}\left( \phi\left( \frac{k+1}{j+1}\right) - \phi\left( \frac{k}{j+1}\right)\right)\\
&\qquad\quad - n^{-2\beta}j^{2\beta}\left( \phi\left( \frac{k+1}{j}\right) - \phi\left( \frac{k}{j}\right)\right)\\
&= n^{-2\beta}\left((j+1)^{2\beta}-j^{2\beta}\right)\left( \phi\left( \frac{k+1}{j+1}\right) - \phi\left( \frac{k}{j+1}\right)\right)\\
&\qquad\quad +n^{-2\beta}j^{2\beta}\left[ \phi\left( \frac{k+1}{j+1}\right) - \phi\left( \frac{k}{j+1}\right) - \phi\left( \frac{k+1}{j}\right) + \phi\left( \frac{k}{j}\right)\right].
\end{align*}
We first show  (\ref{ineq:covpartial1}). Condition $j+3 \le k \le2j+2$ implies that the  interval $\left[ \frac k {j+1}, \frac {k+1} j \right]$ is included in the interval $[1,5]$.
Therefore, using (\ref{ec1}) and  (\ref{ec2}), we deduce that there exists a constant $C>0$ such that for all   $x\in \left[ \frac k {j+1}, \frac {k+1} j \right]$,
\[ 
\left| \phi' \left( x \right)\right| \le    C(k/j)^{\alpha-1}.
\]
and
\[ 
\left| \phi''\left( x \right)\right| \le    C(k/j)^{\alpha-2}.
\]
The  estimate (\ref{ineq:covpartial1}) follows easily from the Mean Value Theorem.

On the other hand   $k\ge 2j+2$ implies that the  interval $\left[ \frac k {j+1}, \frac {k+1} j \right]$ is included in the interval $[2, \infty]$.
Therefore, using (\ref{ec1}) and  (\ref{ec2}), we deduce that there exists a constant $C>0$ such that for all   $x\in \left[ \frac k {j+1}, \frac {k+1} j \right]$,
\[ 
\left| \phi' \left( x \right)\right| \le    C(k/j)^{-\nu}.
\]
and
\[ 
\left| \phi''\left( x \right)\right| \le    C(k/j)^{-\nu-1}.
\]
Therefore,  estimate (\ref{ineq:covpartial2}) follows easily from the Mean Value Theorem. The proof of the lemma is now complete.
\end{proof}
%%================

Last, we have two technical results that have been used in the proofs of Theorems  \ref{thm1} and \ref{Cor1}.  For a fixed integer $n$ and nonnegative real $t_{1},t_{2}$, note that the notation of Section \ref{Notationextra} gives
\[\mathbb{E}[\Delta X_{\frac{t_{1}}{n}}\Delta X_{\frac{t_{2}}{n}}]=\Ip{\partial_{\frac{t_{1}}{n}},\partial_{\frac{t_{2}}{n}}}_{\hten}.\]
 
\begin{lemma}\label{lem:sum_est1}
Assume $X$ satisfies (H.1) and (H.2).  Then for any integer $n\ge 2$ and real $T>0$, there is a constant  $C$  is a constant which depends on $T$ and the covariance of $X$, such that
\begin{equation}\label{ineq:partialjkbound}
\sup_{0\le k\le  \lfloor nT \rfloor-1}\sum_{j=0}^{\lfloor nT \rfloor -1}\left|\left< \partial_{\frac jn}, \partial_{\frac kn}\right>_\hten\right| \le Cn^{-\alpha}.
\end{equation}
\end{lemma}

\begin{proof}     In view of the estimate (\ref{ineq:varpartial}), we can assume that $n\ge 6$ and $4\le j+3 \le k$ or $4\le k+3 \le j$. If  $4\le j+3 \le k$, from the estimates (\ref{ineq:covpartial1}) and (\ref{ineq:covpartial2}), we deduce
\[
\left|\left< \partial_{\frac jn}, \partial_{\frac kn}\right>_\hten\right| \le C n^{-2\beta} j^{2\beta-2}.
\]
Summing in the index $j$ we get the desired result, because $2\beta-1 \le 0$ and $2\beta \ge \alpha$. On the other hand, if  $4\le k+3 \le j\le 2k+2$, the estimates (\ref{ineq:covpartial1})  yields
\[
\left|\left< \partial_{\frac jn}, \partial_{\frac kn}\right>_\hten\right| \le C n^{-2\beta} k^{2\beta-\alpha} j^{\alpha-2} \le C n^{-\alpha} j^{\alpha-2},
\]
which gives the desired estimate. Finally,  if   $4\le k+3$ and $2k+2 \le j$, the estimate    (\ref{ineq:covpartial2}) yields
\[
\left|\left< \partial_{\frac jn}, \partial_{\frac kn}\right>_\hten\right| \le C n^{-2\beta} k^{2\beta+\nu -2} j^{-\nu}.
\]
If  $\alpha +\nu -2 \le 0$,  then summing the above estimate in $j$ we obtain the bound
 \[
Cn^{-2\beta}  k^{2\beta-\alpha + (\alpha +\nu-2)} \le Cn^{-\alpha}.
\]
On the other hand, if 
$\alpha +\nu -2  >0$, then
 \[
C n^{-2\beta} k^{2\beta  +\nu -2} j^{-\nu} \le  C n^{-2\beta} k^{2\beta -\alpha } \left( \frac kj \right) ^{\alpha +\nu -2}  j^{\alpha-2}  \le C
n^{-\alpha} j^{\alpha-2}
\]
and summing in $j$ we get the desired bound.
\end{proof}
%================
\begin{lemma}\label{lem:aux}
Assume that $0<\alpha<1$ and let $n\geq 1$ be an integer. Then, for every $r\in\mathbb{N}$ and $T\geq0$,
\begin{align}\label{ineq:covpartialepsilon}
\sum_{j=0}^{\floor{nT}-1}\Abs{\Ip{\partial_{\frac{j}{n}},\widetilde{\varepsilon}_{\frac{j}{n}}}_{\hten}}^{r}\leq Cn^{-2\beta(r-1)}.
\end{align}
\end{lemma}
%----------
\begin{proof}
By \eqref{eq:R},
\begin{align*}
\Ip{\partial_{\frac{j}{n}},\widetilde{\varepsilon}_{\frac{j}{n}}}_{\hten}
  &=\frac 12\mathbb{E}\left[(X_{\frac{j+1}{n}}-X_{\frac{j}{n}})(X_{\frac{j+1}{n}}+X_{\frac{j}{n}})\right]=\frac{1}{2}\mathbb{E}\left[X_{\frac{j+1}{n}}^2-X_{\frac{j}{n}}^2\right]=\phi(1)\Psi_{n}(j),
\end{align*}
where 
\begin{align*}
\Psi_{n}(j)
  &=\left(\left(\frac{j+1}{n}\right)^{2\beta}-\left(\frac{j}{n}\right)^{2\beta}\right).
\end{align*}
We can easily show that $\Psi_{n}(j)\leq  C n^{-2\beta}$, and hence, 
\begin{align*}
\sum_{j=0}^{\floor{nT}-1}\Abs{\Ip{\partial_{\frac{j}{n}},\widetilde{\varepsilon}_{\frac{j}{n}}}_{\hten}}^r
  &=\phi(1)^{r}\sum_{j=0}^{\floor{nT}-1}\Psi_{n}(j)^r\leq  Cn^{-2\beta(r-1)}\sum_{j=0}^{\floor{nT}-1}\Psi_{n}(j).
\end{align*}
Since the right-hand side of the last  inequality is a telescopic sum, we get
\begin{align*}
\sum_{j=0}^{\floor{nT}-1}\Abs{\Ip{\partial_{\frac{j}{n}},\widetilde{\varepsilon}_{\frac{j}{n}}}_{\hten}}^r
  &\leq  Cn^{-2\beta(r-1)}\left(\frac{\floor{nT}}{n}\right)^{2\beta}.
\end{align*}
Relation \eqref{ineq:covpartialepsilon} follows from the previous inequality.
\end{proof}
%========================================================================================================================

\end{document}